\documentclass[brochure,12pt]{bourbaki}% Pour un exposé en français
\usepackage[matrix,arrow]{xy}
\usepackage{amssymb,amsfonts,amsmath,footnote}
\usepackage{mathrsfs}
\usepackage[T1]{fontenc}
\usepackage[francais]{babel}
\date{Octobre 2017}
\bbkannee{70\`eme ann\'ee, 2017-2018}
\bbknumero{1139}
\title{La th\'eorie de Hodge des bimodules de Soergel}
\subtitle{d'apr\`es Soergel et Elias--Williamson}
\author{Simon RICHE}
\address{Laboratoire de Math\'ematiques Blaise Pascal\\
UMR CNRS 6620\\
Universit\'e Clermont Auvergne\\ 
Campus Universitaire des C\'ezeaux\\
3, place Vasarely\\ 
TSA 60026 -- CS 60026\\ 
F--63178 Aubi\`ere Cedex}
\email{simon.riche@uca.fr}

\newcommand{\frg}{\mathfrak{g}}
\newcommand{\frb}{\mathfrak{b}}
\newcommand{\frh}{\mathfrak{h}}

\newcommand{\Phis}{\Phi_{\mathrm{s}}}
\newcommand{\cH}{\mathcal{H}}
\newcommand{\uH}{\underline{H}}
\newcommand{\cO}{\mathcal{O}}

\newcommand{\B}{\mathsf{B}}
\newcommand{\BS}{\mathsf{BS}}
\newcommand{\scB}{\mathscr{B}}
\newcommand{\uw}{\underline{w}}
\newcommand{\sm}{\mathsf{m}}
\newcommand{\bG}{\mathbf{G}}
\newcommand{\bB}{\mathbf{B}}

\newcommand{\R}{\mathbb{R}}
\newcommand{\C}{\mathbb{C}}
\newcommand{\Z}{\mathbb{Z}}
\newcommand{\D}{\mathbb{D}}
\newcommand{\id}{\mathrm{id}}

\newcommand{\vv}{\mathsf{v}}

\DeclareMathOperator{\Hom}{Hom}
\DeclareMathOperator{\End}{End}
\newcommand{\Kb}{K^{\mathrm{b}}}
\newcommand{\Db}{D^{\mathrm{b}}}

\newcommand{\simto}{\xrightarrow{\sim}}

\begin{document}
\maketitle

\noindent{\bf INTRODUCTION}

%\bigskip

%----------------------------------------------------------
\subsection{Bimodules de Soergel}
\label{ss:intro-bimodules}
%----------------------------------------------------------

Les \emph{bimodules de Soergel} sont certains bimodules gradu\'es sur des alg\`ebres de polyn\^omes, associ\'es \`a un syst\`eme de Coxeter $(W,S)$ et une repr\'esentation $V$ de $W$.
% satisfaisant une condition technique. Ces bimodules forment une cat\'egorie additive mono\"idale, dont l'anneau de Grothendieck scind\'e s'identifie canoniquement \`a l'alg\`ebre de (Iwahori--)Hecke $\cH_{(W,S)}$ de $(W,S)$.

Ces bimodules ont \'et\'e introduits au d\'ebut des ann\'ees 90 dans les travaux de Soergel~\cite{soergel-kat, soergel-HC}, dans le cas particulier o\`u $W$ est le groupe de Weyl d'un groupe alg\'ebrique semisimple complexe $G$, et $V$ est l'alg\`ebre de Lie d'un tore maximal de $G$. Dans ce cas Soergel montre que cette cat\'egorie est \'equivalente \`a la cat\'egorie des complexes semisimples $B$-\'equivariants (pour la t-structure perverse) sur la vari\'et\'e de drapeaux $G/B$ de $G$ (o\`u $B \subset G$ est un sous-groupe de Borel) ; on en d\'eduit que
\begin{enumerate}
\item
\label{eqn:proprietes-Soergel-1}
l'anneau de Grothendieck scind\'e de cette cat\'egorie est canoniquement isomorphe \`a l'alg\`ebre de (Iwahori--)Hecke $\cH_{(W,S)}$ de $(W,S)$ ;
\item
\label{eqn:proprietes-Soergel-2}
les bimodules de Soergel ind\'ecomposables (\`a d\'ecalage pr\`es de la graduation) sont param\'etr\'es par $W$ (on notera $\B_w$ le bimodule associ\'e \`a $w \in W$) ;
\item
\label{eqn:proprietes-Soergel-3}
%ces objets ind\'ecomposables s'interpr\`etent comme la cohomologie d'intersection $B$-\'equivariante des vari\'et\'es de Schubert (c'est-\`a-dire des adh\'erences de $B$-orbits) dans $G/B$ ; en particulier 
les classes de ces objets dans l'anneau de Grothendieck, identifi\'e \`a $\cH_{(W,S)}$, forment la \emph{base de Kazhdan--Lusztig} de $\cH_{(W,S)}$.
\end{enumerate}
En utilisant ces propri\'et\'es, Soergel obtient alors une nouvelle preuve de la conjecture de Kazhdan--Lusztig~\cite{kl} sur les multiplicit\'es dans la cat\'egorie $\mathcal{O}$ de l'alg\`ebre de Lie semisimple complexe duale de $G$ au sens de Langlands, prouv\'ee pr\'ec\'edemment et par d'autres m\'ethodes\footnote{On renvoie \`a~\cite{springer} pour une pr\'esentation de ces travaux.} par Be{\u\i}linson--Bernstein~\cite{bb} et Brylinsky--Kashiwara~\cite{bk}.

Dans un travail ult\'erieur~\cite{soergel-bim} (apr\`es des premiers r\'esultats obtenus dans~\cite{soergel-HC}), Soergel a d\'efini\footnote{Dans~\cite{soergel-bim}, Soergel utilise le terme ``speziellen Bimoduln'' pour ces objets. Les premi\`eres occurences du terme ``bimodules de Soergel'' dans la litt\'erature semblent \^etre dans~\cite{rouquier-preprint} et~\cite{khovanov}.} et \'etudi\'e ces bimodules pour un syst\`eme de Coxeter arbitraire et une repr\'esentation de $W$ satisfaisant une certaine condition technique\footnote{Wolfgang Soergel m'a fait savoir qu'il tenait \`a remercier particuli\`erement Patrick Polo pour ses contributions dans la gen\`ese de l'article~\cite{soergel-bim}.}. Dans ce cadre il n'existe pas de ``g\'eom\'etrie'' associ\'ee, de sorte que les arguments utilis\'es sont n\'ecessairement alg\'ebriques. Il montre dans cet article que les propri\'et\'es~\eqref{eqn:proprietes-Soergel-1} et~\eqref{eqn:proprietes-Soergel-2} ci-dessus sont vraies dans cette g\'en\'eralit\'e. Par contre, la propri\'et\'e~\eqref{eqn:proprietes-Soergel-3} s'av\`ere \^etre beaucoup plus subtile, et n'est \'enonc\'ee que comme une conjecture (sous l'hypoth\`ese que le corps de base est de caract\'eristique $0$). Soergel remarque que cette conjecture impliquerait la positivit\'e des polyn\^omes de Kazhdan--Lusztig (comme conjectur\'e par Kazhdan--Lusztig~\cite{kl}), c'est-\`a-dire une solution \`a l'un des probl\`emes centraux dans la combinatoire des groupes de Coxeter.

Depuis ces travaux fondateurs, les bimodules de Soergel (sous diff\'erentes formes) se sont r\'ev\'el\'es \^etre des outils extr\`emement utiles en th\'eorie des repr\'esentations (voir~\cite{soergel-HC, soergel-andersen, ww, dodd, by, rw} pour quelques exemples), car ils permettent souvent de faire un lien avec la g\'eom\'etrie d'une vari\'et\'e de drapeaux appropri\'ee. Mais dans tous les cas la preuve de certaines de leurs propri\'et\'es sort du cadre alg\'ebrique de leur d\'efinition, et repose sur la g\'eom\'etrie.

%----------------------------------------------------------
\subsection{Les r\'esultats d'Elias--Williamson}
\label{ss:intro-resultats}
%----------------------------------------------------------

L'objet principal de cet expos\'e est d'expliquer la preuve, due \`a Elias--William\-son~\cite{ew1}, de la conjecture de Soergel \'evoqu\'ee au~\S\ref{ss:intro-bimodules}. Soergel remarque dans~\cite{soergel-bim} que, si on cherche \`a d\'emontrer la conjecture par r\'ecurrence sur la longueur de l'\'el\'ement de $W$ consid\'er\'e, on doit v\'erifier que certaines formes bilin\'eaires sur des espaces de morphismes sont non d\'eg\'en\'er\'ees. L'intuition fondamentale de Elias--Williamson est que cette propri\'et\'e n'est que la partie \'emerg\'ee d'un iceberg beaucoup plus profond : les bimodules de Soergel poss\`edent toute une ``th\'eorie de Hodge'' qu'il faut construire en m\^eme temps qu'on d\'emontre la conjecture de Soergel.

Plus pr\'ecis\'ement, Elias--Williamson se placent dans le cadre d'une repr\'esentation $V$ \emph{r\'eelle} de $(W,S)$ qui poss\`ede la propri\'et\'e technique de Soergel et une autre condition de ``positivit\'e''\footnote{La preuve d'Elias--Williamson repose sur l'utilisation de l'ordre sur $\mathbb{R}$, et donc ne s'applique pas \`a d'autres corps de coefficients. Les sp\'ecialistes semblent douter de la v\'eracit\'e de la conjecture de Soergel sur d'autres corps de caract\'eristique $0$. (Divers contre-exemples sont connus en caract\'eristique positive.)}. (Chaque groupe de Coxeter poss\`ede une repr\'esentation de ce type.) On pose $R:=\mathrm{Sym}(V^*)$, qu'on consid\`ere comme un anneau gradu\'e engendr\'e en degr\'e $2$. Ils consid\`erent alors un certain \'element $\rho \in V^*=R^2$ et remarquent que si $w \in W$ et si la conjecture de Soergel est connue pour $w$ (c'est-\`a-dire si la classe dans le groupe de Grothendieck scind\'e du bimodule ind\'ecomposable $\B_w$ associ\'e \`a $w$ est l'\'el\'ement de la base de Kazhdan--Lusztig associ\'e \`a $w$) alors $\B_w$ poss\`ede une forme bilin\'eaire sym\'etrique ``invariante'' $\langle -, - \rangle_{\B_w} : \B_w \times \B_w \to R$ canonique, unique \`a un scalaire dans $\R_{>0}$ pr\`es. On notera de m\^eme la forme bilin\'eaire sym\'etrique $(\B_w \otimes_R \R) \times (\B_w \otimes_R \R) \to \R$ induite par $\langle -, - \rangle_{\B_w}$. Ils d\'emontrent simultan\'ement, par r\'ecurrence sur la longueur de $w$, que :
\begin{enumerate}
\item
\label{it:intro-prop-1}
la conjecture de Soergel est vraie pour $w$;
\item
\label{it:intro-prop-2}
$\B_w \otimes_R \R$ v\'erifie le th\'eor\`eme de Lefschetz difficile, au sens o\`u pour tout $i \geq 0$ la multiplication par $\rho^i$ induit un isomorphisme $(\B_w \otimes_R \R)^{-i} \simto (\B_w \otimes_R \R)^i$ entre les parties de degr\'e $-i$ et $i$;
\item
\label{it:intro-prop-3}
la paire $(\B_w \otimes_R \R, \langle -, - \rangle_{\B_w})$ v\'erifie les relations de Hodge--Riemann, au sens o\`u pour tout $i \geq 0$ la restriction de la forme $(x,y) \mapsto \langle x, \rho^i \cdot y \rangle_{\B_w}$ sur $(\B_w \otimes_R \R)^{-i}$ aux \'el\'ements primitifs (c'est-\`a-dire annul\'es par $\rho^{i+1}$) est $(-1)^{\frac{\ell(w)-i}{2}}$-d\'efinie\footnote{Ici on dit qu'une forme bilin\'eaire sym\'etrique est $(+1)$-d\'efinie, resp.~$(-1)$-d\'efinie, si elle est d\'efinie positive, resp.~d\'efinie n\'egative.} (o\`u $\ell$ d\'esigne la fonction de longueur sur $W$).
\end{enumerate}

Notons que la propri\'et\'e~\eqref{it:intro-prop-3} est l'\'etape cruciale pour d\'emontrer~\eqref{it:intro-prop-1}, mais que d'un autre c\^ot\'e cette propri\'et\'e n'a un sens pr\'ecis qu'une fois que~\eqref{it:intro-prop-1} est d\'emontr\'ee (pour que la forme bilin\'eaire $\langle -, - \rangle_{\B_w}$ soit fix\'ee). Les preuves de ces deux propri\'et\'es sont donc n\'ecessairement imbriqu\'ees. La propri\'et\'e~\eqref{it:intro-prop-2} est elle impliqu\'ee par~\eqref{it:intro-prop-3}, mais joue un r\^ole crucial dans la r\'ecurrence.

Dans le cas o\`u $W$ est le groupe de Weyl d'un groupe alg\'ebrique complexe $G$, et o\`u $V$ est l'alg\`ebre de Lie d'un tore maximal, $\B_w \otimes_R \R$ s'identifie \`a la cohomologie d'intersection de la vari\'et\'e de Schubert associ\'ee \`a $W$. Dans ce cas, les propri\'et\'es~\eqref{it:intro-prop-2} et~\eqref{it:intro-prop-3} d\'ecoulent du th\'eor\`eme de Lefschetz difficile et des relations de Hodge--Riemann ``classiques'', appliqu\'es \`a ce cas particulier. On renvoie \`a~\cite{ew-survey, williamson-survey} pour plus de d\'etails.

%----------------------------------------------------------
\subsection{Structure de la preuve}
\label{ss:structure-preuve}
%----------------------------------------------------------

La structure de la preuve de~\cite{ew1} est inspir\'ee par la preuve du th\'eor\`eme de d\'ecomposition donn\'ee r\'ecemment par de Cataldo--Migliorini. (Voir~\cite{dcm, williamson-bourbaki} pour des pr\'esentations de cette preuve. Notons cependant que la ``traduction'' de ces id\'ees dans le langage alg\'ebrique des bimodules de Soergel est hautement non triviale !) On a essay\'e de ne pr\'esenter que les parties indispensables de cette preuve, quitte \`a omettre certains r\'esultats interm\'ediaires int\'eressants mais qui peuvent \^etre contourn\'es (voir par exemple la Remarque~\ref{rema:positivite-inverse}).

Comme expliqu\'e ci-dessus, la preuve proc\`ede par r\'ecurrence sur la longueur de l'\'el\'ement consid\'er\'e. On fixe donc $w \in W \smallsetminus \{e\}$, on \'ecrit $w=ys$ avec $s \in S$ et $\ell(y)<\ell(w)$, et on suppose toutes les propri\'et\'es voulues pour les \'el\'ements de longueur $\leq \ell(y)$. Comme sugg\'er\'e au~\S\ref{ss:intro-resultats}, pour d\'emontrer la conjecture de Soergel pour $w$ il faut v\'erifier que certaines formes bilin\'eaires sur des espaces $\Hom_{\scB}(\B_z, \B_y \otimes_R \B_s)$ (o\`u $\ell(z) \leq \ell(y)$) sont non d\'eg\'en\'er\'ees. Pour ce faire, Elias--Williamson remarquent que $\Hom_{\scB}(\B_z, \B_y \otimes_R \B_s)$ s'injecte (par une isom\'etrie pour des formes bilin\'eaires appropri\'ees) dans le sous-espace des \'el\'ements primitifs dans $\bigl( (\B_y \otimes_R \B_s) \otimes_R \R \bigr)^{-\ell(z)}$. Puisque la restriction \`a un sous-espace d'une forme bilin\'eaire sym\'etrique d\'efinie positive ou n\'egative est non d\'eg\'en\'er\'ee, ceci ram\`ene la preuve de la conjecture de Soergel pour $w$ \`a un r\'esultat de type ``Hodge--Riemann'' pour $\B_y \otimes_R \B_s$. Il est par ailleurs facile de voir que le th\'eor\`eme de Lefschetz difficile et les relations de Hodge--Riemann pour $\B_w$ d\'ecoulent \'egalement de ce r\'esultat.

Pour d\'emontrer ce r\'esultat, Elias--Williamson consid\`erent l'op\'erateur de Lefschetz ``d\'eform\'e''
\[
L^{y,s}_\zeta := \rho \cdot \id_{\B_y \otimes \B_s} + \id_{\B_y} \otimes (\zeta \rho \cdot \id_{\B_s}) : \B_y \otimes_R \B_s \to \B_y \otimes_R \B_s(2)
\]
pour $\zeta \in \R_{\geq 0}$. (Ici, $(2)$ d\'esigne le foncteur de d\'ecalage de la graduation par $2$ vers la gauche.)
Ils remarquent tout d'abord que les relations de Hodge--Riemann pour $\B_y$ impliquent que $L^{y,s}_\zeta$ v\'erifie le r\'esultat voulu si $\zeta \gg 0$. En utilisant le fait que ce r\'esultat peut s'\'enoncer en terme de la signature de certaines formes bilin\'eaires, et que la signature ne peut pas varier dans une famille continue de formes bilin\'eaires sym\'etriques non d\'eg\'en\'er\'ees, ceci ram\`ene la preuve du r\'esultat voulu \`a v\'erifier que $L^{y,s}_\zeta$ v\'erifie le th\'eor\`eme de Lefschetz difficile pour tout $\zeta \geq 0$.

Dans leur preuve de l'\'enonc\'e g\'eom\'etrique correspondant, de Cataldo--Migliorini utilisent le th\'eor\`eme de Lefschetz faible pour une section hyperplane. Cet \'enonc\'e n'a aucun analogue dans le monde des bimodules de Soergel, et Elias--Williamson utilisent \`a la place des diff\'erentielles apparaissant dans les ``complexes de Rouquier'', qui permettent de factoriser l'op\'erateur $\rho \cdot \id_{\B_y}$ en une compos\'ee $\B_y \to D(1) \to \B_y(2)$, o\`u $D$ est une somme de termes de la forme $\B_z$ avec $\ell(z) < \ell(y)$ qui v\'erifie une version appropri\'ee des relations de Hodge--Riemann pour une forme convenable. Notons que cet argument impose de consid\'erer \'egalement les relations de Hodge--Riemann pour l'op\'erateur $L_\zeta$ sur $\B_z \otimes_R \B_s$ (quand $\ell(zs)>\ell(z)$) comme une hypoth\`ese de r\'ecurrence, en plus des propri\'et\'es consid\'er\'ees au~\S\ref{ss:intro-resultats}.

%----------------------------------------------------------
\subsection{Applications}
%----------------------------------------------------------

La preuve de la conjecture de Soergel a deux applications tr\`es importantes, d\'ej\`a sugg\'er\'ees au~\S\ref{ss:intro-bimodules}. La premi\`ere est qu'elle permet de donner une preuve alg\'ebrique de la conjecture de Kazhdan--Lusztig concernant les multiplicit\'es des modules simples dans les modules de Verma d'une alg\`ebre de Lie semi-simple complexe.
%(La preuve originale de cette conjecture, due \`a Be{\u\i}linson--Bernstein et Brylinsky--Kashiwara ind\'ependamment, repose sur l'utilisation des faisceaux pervers sur les vari\'et\'es de drapeaux.) 
La seconde application (sugg\'er\'ee par Soergel dans~\cite{soergel-bim}) est la preuve d'une autre conjecture de Kazhdan--Lusztig affirmant que les coefficients des \emph{polyn\^omes de Kazhdan--Lusztig} (des polyn\^omes apparaissant dans diverses formules de caract\`eres en th\'eorie de Lie) sont positifs ou nuls. Cette conjecture \'etait connue quand $W$ est le groupe de Weyl d'un groupe alg\'ebrique (ou plus g\'en\'eralement d'un groupe de Kac--Moody) ; mais dans le cas g\'en\'eral elle constituait un des probl\`emes centraux du domaine depuis sa formulation en 1979.

%----------------------------------------------------------
\subsection{Prolongements}
%----------------------------------------------------------

Du point de vue d'Elias--Williamson, les r\'esultats \'enonc\'es au~\S\ref{ss:intro-resultats} forment la th\'eorie de Hodge ``globale'' des bimodules de Soergel. Dans des articles ult\'erieurs~\cite{williamson, ew2}, ils ont d\'evelopp\'e des versions ``locale'' et ``relative'' de cette th\'eorie, qui ont \'egalement des applications importantes en th\'eorie de Lie. Pour des raisons de place, ces r\'esultats (et leurs applications) ne sont \'evoqu\'es que bri\`evement, \`a la fin de ces notes.

%----------------------------------------------------------
\subsection{Contenu des notes}
%----------------------------------------------------------

Dans la Partie~\ref{sec:bimodules}, on rappelle les r\'esultats de base sur les bimodules de Soergel, suivant~\cite{soergel-bim}. Dans la Partie~\ref{sec:global} on pr\'esente les r\'esultats principaux de~\cite{ew1} et leurs preuves. Dans la Partie~\ref{sec:applications} on expose les applications de ces r\'esultats. Enfin, dans la Partie~\ref{sec:variantes} on \'evoque bri\`evement les r\'esultats de~\cite{williamson} et~\cite{ew2}.

%----------------------------------------------------------
\subsection{Remerciements}
%----------------------------------------------------------

Je remercie Geordie Williamson pour de nombreuses discussions autour de ces travaux et de r\'esultats connexes, avant et pendant la pr\'eparation de cet expos\'e, ainsi que pour ses remarques sur une version pr\'eliminaire de ces notes. Je remercie \'egalement Wolfgang Soergel pour l'autorisation de reproduire sa preuve de la Proposition~\ref{prop:modules-Soergel} et pour ses remarques sur une version pr\'eliminaire de ce texte. Enfin, je remercie V.~Le Dret pour sa relecture tr\`es attentive.

Ce travail a b\'en\'efici\'e du soutien de la bourse~ANR-13-BS01-0001-01 de l'Agence Nationale de la Recherche. This project has received funding from the European Research Council (ERC) under the European Union's Horizon 2020 research and innovation programme (grant agreement No. 677147).

%%%%%%%%%%%%%%%%%%%%%%%%%%%%%%%
\section{Bimodules de Soergel}
\label{sec:bimodules}
%%%%%%%%%%%%%%%%%%%%%%%%%%%%%%%

%------------------------------------------------------------
\subsection{Syst\`emes de Coxeter et alg\`ebres de Hecke}
\label{ss:def-Coxeter}
%------------------------------------------------------------

%Soit $S$ un ensemble fini, et $(m_{st})_{(s,t) \in S}$ une famille d'\'el\'ements de $\Z_{\geq 1} \cup \{\infty\}$ telle que $m_{ss}=1$ pour tout $s \in S$ et $m_{st}=m_{ts} \geq 2$ si $s \neq t$. Le \emph{groupe de Coxeter} $W$ associ\'e \`a ces donn\'ees est le groupe dont une pr\'esentation est
%\[
%W = \langle S \mid (st)^{m_{st}}=1 \text{ pour tous $s,t \in S$} \rangle.
%\]
%(Si $m_{st}=\infty$, on interpr\`ete la relation $(st)^{m_{st}}=1$ comme une absence de relation.)

Soit $W$ un groupe, et soit $S \subset W$ un sous-ensemble constitu\'e d'\'el\'ements d'ordre $2$, qu'on supposera fini. Pour tous $s,t \in S$ (non n\'ecessairement distincts), on note $m_{st} \in \Z_{\geq 1} \cup \{\infty\}$ l'ordre de $st$ dans $W$. On dit que $(W,S)$ est un \emph{syst\`eme de Coxeter} si le morphisme canonique du groupe de pr\'esentation
\[
\langle S \mid (st)^{m_{st}}=1 \text{ pour tous $s,t \in S$ tels que $m_{st} \neq \infty$} \rangle
\]
dans $W$ est un isomorphisme. (On dit alors parfois -- comme ici dans l'introduction -- que $W$ est un \emph{groupe de Coxeter}, mais on \'evitera cette terminologie impropre dans le reste de ces notes.) Notons que pour tout ensemble $S$ et tous \'el\'ements $m_{st} \in \Z_{\geq 1} \cup \{\infty\}$ tels que $m_{ss}=1$ et $m_{st}=m_{ts} \geq 2$ pour $s \neq t$, il existe un syst\`eme de Coxeter de g\'en\'erateurs $S$ et tel que $m_{st}$ est l'ordre de $st$ pour tous $s,t \in S$ ; voir par exemple~\cite[\S V.4.3]{bourbaki}.

Si $w \in W$, on appellera \emph{d\'ecomposition r\'eduite} de $w$ toute \'ecriture de $w$ en produit d'\'el\'ements de $S$ de longueur minimale. Cette longueur est appel\'ee la longueur de $w$, et not\'ee $\ell(w)$. L'ensemble $W$ est muni d'un ordre naturel appel\'e \emph{ordre de Bruhat}, qu'on notera $\leq$, et pour lequel la fonction $\ell : W \to \Z_{\geq 0}$ est strictement croissante.

Ces groupes apparaissent ``naturellement'' en th\'eorie de Lie de la fa{\c c}on suivante.
%les contextes suivants :
%\begin{enumerate}
%\item
%Si $W \subset \mathrm{GL}_n(\R)$ est un sous-groupe fini engendr\'e par des r\'eflexions (c'est-\`a-dire des endomorphismes agissant par $\id$ sur un hyperplan, et par $-\id$ sur une droite suppl\'ementaire), alors $W$ est un groupe de Coxeter pour un certain choix de $S$.
%\item
%le cadre suivant.
Soit $G$ un groupe alg\'ebrique r\'eductif connexe sur un corps alg\'ebriquement clos, soit $B \subset G$ un sous-groupe de Borel, et soit $T \subset B$ un tore maximal. Soit $W=N_G(T)/T$ le groupe de Weyl de $(G,T)$. Soit \'egalement $\Phi$ le syst\`eme de racines de $G$, soit $\Phi^+ \subset \Phi$ le syst\`eme positif constitu\'e des racines de $B$, et soit $\Phis$ la base de $\Phi$ associ\'ee. \`A toute racine $\alpha$ on peut associer une \emph{coracine} $\alpha^\vee \in X_*(T)$, et une \emph{r\'eflexion} $s_\alpha \in W$ telle que $s_\alpha$ agit sur $T$ par $s_\alpha(t) = t \cdot \alpha^\vee(\alpha(t)^{-1})$. Si $S = \{s_\alpha : \alpha \in \Phis\}$, alors $(W,S)$ est un syst\`eme de Coxeter. Si on consid\`ere une construction similaire plus g\'en\'eralement pour les \emph{groupes de Kac--Moody} associ\'es \`a des matrices de Cartan g\'en\'eralis\'ees (voir~\cite{tits, kumar}), on obtient de cette mani\`ere tous les groupes de Coxeter \emph{cristallographiques}, c'est-\`a-dire tels que $m_{st} \in \{2,3,4,6,\infty\}$ pour tous $s \neq t$.

Si $(W,S)$ est un syst\`eme de Coxeter, l'\emph{alg\`ebre de Hecke} associ\'ee est la $\Z[\vv,\vv^{-1}]$-alg\`ebre $\cH_{(W,S)}$ (o\`u $\vv$ est une ind\'etermin\'ee) engendr\'ee par des \'el\'ements $\{H_s : s \in S\}$ soumis aux relations suivantes :
\begin{itemize}
\item
\emph{relations quadratiques} : $(H_s)^2 = 1 + (\vv^{-1}-\vv) H_s$ pour $s \in S$;
\item
\emph{relations de tresse} : si $s, t \in S$ avec $s \neq t$ et $m_{st} \neq \infty$, alors
\[
\underbrace{H_s H_t \cdots}_{\text{$m_{st}$ termes}} = \underbrace{H_t H_s \cdots }_{\text{$m_{st}$ termes}}.
\]
%(avec $m_{st}$ termes de chaque c\^ot\'e de l'\'egalit\'e).
\end{itemize}
Les relations de tresses et le ``lemme de Matsumoto'' (voir~\cite[\S IV.1.5, Proposition~5]{bourbaki}) assurent que si $w \in W$ et si $w=s_1 \cdots s_n$ est une expression r\'eduite, alors l'\'el\'ement $H_w := H_{s_1} \cdots H_{s_n}$ ne d\'epend que de $w$. De plus, les \'el\'ements $\{H_w : w \in W\}$ forment une $\Z[\vv,\vv^{-1}]$-base de $\cH_{(W,S)}$ (voir~\cite[\S IV.2, Exercice 23]{bourbaki}).

%de base $\{H_w : w \in W\}$ et dont la multiplication est d\'etermin\'ee par
%\[
%H_s \cdot H_w = \begin{cases}
%H_{sw} & \text{si $sw>w$;} \\
%(v-v^{-1})H_w + H_{sw} & \text{si $sw<w$.}
%\end{cases}
%\]
Les relations quadratiques montrent que les \'el\'ements $H_s$ sont inversibles dans $\cH_{(W,S)}$. Il en d\'ecoule que cette propri\'et\'e est vraie pour tous les \'el\'ements $H_w$, de sorte qu'on peut d\'efinir une involution d'anneaux $d$ de $\cH_{(W,S)}$ en posant
\[
d(\vv)=\vv^{-1} \quad \text{et} \quad d(H_w) = (H_{w^{-1}})^{-1}.
\]
Kazhdan--Lusztig ont construit dans~\cite{kl} une base remarquable de $\cH_{(W,S)}$ de la fa{\c c}on suivante :
pour tout $w \in W$, il existe un unique \'el\'ement $\uH_w$ qui v\'erifie
\[
d(\uH_w) = \uH_w \quad \text{et} \quad \uH_w \in H_w + \bigoplus_{y \in W} \vv\Z[\vv] \cdot H_y ;
\]
alors $\{\uH_w : w \in W\}$ forme une base de $\cH_{(W,S)}$. (Voir \'egalement~\cite[Theorem~2.1]{soergel} pour une preuve plus simple de ce r\'esultat).

Les \emph{polyn\^omes de Kazhdan--Lusztig} sont les \'el\'ements $\{h_{y,w} : y,w \in W\}$ de $\Z[\vv]$ d\'etermin\'es par
\[
\uH_w = \sum_{y \in W} h_{y,w} \cdot H_y.
\]
(On peut facilement v\'erifier que si $h_{y,w} \neq 0$ alors $y \leq w$, et que $h_{w,w}=1$.) 

Ci-dessous on consid\'erera \'egalement les constantes de structure $\{\mu_{x,y}^z : x,y,z \in W\}$ de la multiplication de $\cH_{(W,S)}$ dans cette base, d\'efinies par
\[
\uH_x \cdot \uH_y = \sum_{z \in W} \mu_{x,y}^z \cdot \uH_z.
\]
Il est facile de voir que si $x \in W$ et $s \in S$ satisfont $xs>x$, alors pour tout $z \in W$ on a
\begin{equation}
\label{eqn:BwBs-perverse}
\mu_{x,s}^z \in \Z.
\end{equation}

%--------------------------------------------------------------
\subsection{Repr\'esentations r\'eflexion fid\`eles}
\label{ss:ref-fidele}
%--------------------------------------------------------------

Soit $(W,S)$ un syst\`eme de Coxeter, et soit $V$ une repr\'esentation de $W$ de dimension finie sur un corps de caract\'eristique impaire ou nulle. On pose $T:=\{wsw^{-1} : w \in W, \, s \in S\}$. Suivant Soergel~\cite{soergel-bim}, on dit que $V$ est \emph{r\'eflexion fid\`ele} si
\begin{enumerate}
\item
$V$ est fid\`ele;
\item
pour tout $x \in W$, $V^x:=\{v \in V \mid x \cdot v = v\}$ est un hyperplan de $V$ si et seulement si $x \in T$.
%il existe $w \in W$ et $s \in S$ tels que $x=wsw^{-1}$.
\end{enumerate}
(Les \'el\'ements de $T$ sont souvent appel\'es \emph{r\'eflexions}. Notons que nos hypoth\`eses assurent que ces \'el\'ements agissent sur $V$ comme des r\'eflexions, c'est-\`a-dire avec un hyperplan de points fixes et une droite propre de valeur propre associ\'ee $-1$.)

Deux exemples importants de repr\'esentations r\'eflexion fid\`eles sont fournis par la proposition suivante. (Le cas~\eqref{it:ref-fidele-soergel} est important car il montre que tout syst\`eme de Coxeter admet une repr\'esentation r\'eflexion fid\`ele ; le cas~\eqref{it:ref-fidele-cartan} est celui qui apparait naturellement en th\'eorie de Lie.)

\begin{prop}
\label{prop:exemples-ref-fidele}
\begin{enumerate}
\item
\label{it:ref-fidele-soergel}
Soit $(W,S)$ un syst\`eme de Coxeter, et soit $V$ un $\R$-espace vectoriel de dimension finie muni d'une famille libre $(e_s)_{s \in S}$ de vecteurs et d'une famille libre $(e_s^\vee)_{s \in S}$ de formes lin\'eaires telles que
\[
\langle e_t, e_s^\vee \rangle = -2\cos(\pi/m_{st}) \quad \text{pour tous $s,t \in S$}
\]
(o\`u par convention $\cos(\pi/\infty)=1$).
Alors la formule $s \cdot v = v - \langle v, e_s^\vee \rangle e_s$ d\'efinit une repr\'esentation de $W$ sur $V$ qui est r\'eflexion fid\`ele.
\item
\label{it:ref-fidele-cartan}
Soit $A$ une matrice de Cartan g\'en\'eralis\'ee sym\'etrisable\footnote{Cette hypoth\`ese n'est probablement pas n\'ecessaire, mais nous n'avons pas pu trouver de preuve qui l'\'evite.} dont les lignes et colonnes sont param\'etr\'ees par un ensemble fini $I$, et soit $(\frh, \{\alpha_i : i \in I\}, \{\alpha_i^\vee : i \in I\})$ une r\'ealisation de $A$ sur $\R$ au sens de Kac (voir~\cite[Definition~1.1.2]{kumar}). Soit $(W,S)$ le syst\`eme de Coxeter associ\'e (voir~\cite[Definition~1.3.1 \& Proposition~1.3.21]{kumar}). Alors la repr\'esentation de $W$ sur $\frh^*$ est r\'eflexion fid\`ele.
\end{enumerate}
\end{prop}

\begin{proof}
Le cas~\eqref{it:ref-fidele-soergel} est d\'emontr\'e dans~\cite[Proposition~2.1]{soergel-bim}. (La condition que $V$ est de dimension minimale impos\'ee dans~\cite{soergel-bim} n'est pas n\'ecessaire.) Le cas~\eqref{it:ref-fidele-cartan} est probablement bien connu, mais n'est pas trait\'e dans la litt\'erature \`a notre connaissance. Donnons-en donc la preuve (qui est tr\`es similaire \`a celle de~\cite[Proposition~2.1]{soergel-bim}).

Puisque $W$ est d\'efini comme un sous-groupe des automorphismes de $\frh^*$, la repr\'esentation de $W$ sur $\frh^*$ est fid\`ele ; par ailleurs, tous les \'el\'ements de $T$ agissent par des r\'eflexions. Soit maintenant $x \in W$, et supposons que $(\frh^*)^x$ est un hyperplan de $\frh^*$. Comme $W$ est engendr\'e par des r\'eflexions de $\frh^*$, on a $\det(x) \in \{\pm 1\}$.

Supposons tout d'abord que $\det(x)=1$. Choisissons $c \in \frh^*$ tel que $\frh^*=(\frh^*)^x \oplus \R c$. Alors on a $x(c) \in c + (\frh^*)^x$; on note $v_0 = x(c)-c$. Pour tous $v \in (\frh^*)^x$ et $\lambda \in \R$ on a donc $x(v+\lambda c) = v + \lambda c + \lambda v_0$. Puisque $A$ est sym\'etrisable, d'apr\`es~\cite[Proposition~1.5.2]{kumar} il existe une forme bilin\'eaire sym\'etrique non d\'eg\'en\'er\'ee $\langle -,- \rangle$ sur $\frh^*$ qui est $W$-invariante. Pour $v$ et $\lambda$ comme ci-dessus on a alors $\langle v+\lambda c, v+\lambda c \rangle = \langle v + \lambda c + \lambda v_0, v + \lambda c + \lambda v_0\rangle$, c'est-\`a-dire
\[
2\lambda \langle v,v_0 \rangle + \lambda^2 \bigl( 2\langle c,v_0 \rangle + \langle v_0,v_0 \rangle \bigr)=0.
\]
On en d\'eduit que $\langle 2c+v_0, v_0 \rangle=0$ et que $\langle v,v_0 \rangle =0$ pour tout $v \in (\frh^*)^x$. Donc $v_0$ appartient au noyau de notre forme, ce qui implique que $v_0=0$. Donc $c \in (\frh^*)^x$, une contradiction.

On suppose maintenant que $\det(x)=-1$. Soient
\[
D = \{v \in \frh^* \mid \forall i \in I, \ \langle v, \alpha_i^\vee \rangle \geq 0 \}, \quad D^\circ = \{v \in \frh^* \mid \forall i \in I, \ \langle v, \alpha_i^\vee \rangle > 0 \},
\]
et consid\'erons \'egalement le \emph{c\^one de Tits} $C=\bigcup_{y \in W} y(D)$. Rappelons que $C$ est convexe ; voir~\cite[Proposition~1.4.2(c)]{kumar}. 
%Les \'el\'ements de $T$ agissent sur $\frh^*$ par des r\'eflexions, et 
Les hyperplans de r\'eflexion des \'el\'ements de $T$ sont les hyperplans d\'efinis par les coracines r\'eelles positives. Si $v \in D^\circ$, d'apr\`es~\cite[Proposition~1.4.2(a)]{kumar} on a $\mathrm{Stab}_W(v)=\{1\}$. Donc $D^\circ$ et $x(D^\circ)$ n'intersectent aucun de ces hyperplans. Il d\'ecoule de~\cite[Proposition~1.4.2(c)]{kumar} que seul un nombre fini de ces hyperplans s\'eparent $D^\circ$ et $x(D^\circ)$.

Pour une raison similaire, $(\frh^*)^x$ ne peut intersecter aucune r\'egion de la forme $y(D^\circ)$ pour $y \in W$. Si $v \in D^\circ$, le segment joignant $v$ \`a $x(v)$ croise $(\frh^*)^x$. Ce point d'intersection appartient \`a $C$ (par convexit\'e) mais \`a aucun des $y(D^\circ)$ pour $y \in W$ ; donc il appartient \`a un hyperplan de r\'eflexion associ\'e \`a un \'el\'ement de $T$, qui doit s\'eparer $D^\circ$ et $x(D^\circ)$. Ainsi, une partie ouverte de $(\frh^*)^x$ est r\'eunion de ses intersections avec un nombre fini d'hyperplans de r\'eflexion. Ceci implique que $(\frh^*)^x$ co\"incide avec l'un de ces hyperplans, dont nous noterons $t$ la r\'eflexion associ\'ee.
Alors $xt$ est un \'el\'ement de $W$ de d\'eterminant $1$ et dont les points fixes contiennent un hyperplan. Par le cas trait\'e ci-dessus, on a n\'ecessairement $xt=e$, donc $x=t$.
\end{proof}

\begin{rema}
Si $A$ est une matrice de Cartan g\'en\'eralis\'ee sym\'etrisable, et si de plus $a_{ij} a_{ji} \leq 4$ pour tous $i,j \in I$, alors on peut v\'erifier que les repr\'esentations du syst\`eme de Coxeter $(W,S)$ associ\'e donn\'ees par les points~\eqref{it:ref-fidele-soergel} et~\eqref{it:ref-fidele-cartan} de la Proposition~\ref{prop:exemples-ref-fidele} sont isomorphes. Par contre, si $a_{ij} a_{ji} \geq 5$ pour un couple $(i,j)$, ces repr\'esentations ne sont pas isomorphes en g\'en\'eral ; en fait elle n'ont m\^eme pas n\'ecessairement la m\^eme dimension.
\end{rema}

%--------------------------------------------------------------
\subsection{Bimodules de Soergel}
\label{ss:bimod-soergel}
%--------------------------------------------------------------

Fixons un syst\`eme de Coxeter $(W,S)$ et une repr\'esentation r\'eflexion fid\`ele $V$ de $(W,S)$ sur un corps $k$ infini\footnote{Cette condition n'est pas n\'ecessaire, mais elle est impos\'ee dans~\cite{soergel-bim} pour pouvoir identifier $R$ \`a l'alg\`ebre des fonctions polynomiales sur $V$.}. On pose $R:=\mathrm{Sym}(V^*)$, qu'on consid\`ere comme une alg\`ebre gradu\'ee avec les formes lin\'eaires $V^* \subset R$ en degr\'e $2$. Pour tout $s \in S$, on d\'efinit un bimodule gradu\'e $\B_s$ sur $R$ en posant
\[
\B_s := R \otimes_{R^s} R (1),
\]
o\`u $R^s \subset R$ d\'esigne les invariants de $s$ et o\`u $(1)$ est le foncteur de d\'ecalage de la graduation tel que $(M(1))^n = M^{n+1}$, de sorte que $\B_s$ est engendr\'e en degr\'e $-1$. (Ci-dessous, on notera $(n)$ la $n$-i\`eme puissance de $(1)$, pour tout $n \in \Z$, et on utilisera cette notation pour d'autres cat\'egories d'objets $\Z$-gradu\'es.) 

Pour tout mot $\uw=(s_1, \cdots, s_r)$ en $S$ on d\'efinit le \emph{bimodule de Bott--Samelson} $\BS(\uw)$ associ\'e \`a $\uw$ en posant
\[
\BS(\uw) := \B_{s_1} \cdots \B_{s_r}.
\]
(Ici et ci-dessous, on omet le symbole de produit tensoriel $\otimes_R$, et nous \'ecrivons donc $MN$ au lieu de $M \otimes_R N$.) Par convention, si $\uw$ est vide ce produit tensoriel est interpr\'et\'e comme \'egal \`a $R$.

On d\'efinit la cat\'egorie $\scB$ des \emph{bimodules de Soergel} associ\'ee \`a $(W,S)$ et $V$ comme la sous-cat\'egorie pleine de la cat\'egorie des $R$-bimodules gradu\'es dont les objets sont les sommes directes de facteurs directs de bimodules de la forme $\BS(\uw)(n)$ pour $\uw$ un mot en $S$ et $n \in \Z$. Comme $\BS(\uw_1) \BS(\uw_2) = \BS(\uw_1\uw_2)$, cette sous-cat\'egorie est stable par produit tensoriel, et donc munie d'une structure naturelle de cat\'egorie mono\"idale.

Il est facile de voir que le bimodule $\B_s$ est libre et de type fini comme $R$-module \`a droite et comme $R$-module \`a gauche. Il en est donc de m\^eme des bimodules $\BS(\uw)$, et par suite (puisque tout $R$-module gradu\'e projectif de type fini est libre) de tous les objets de $\scB$. On en d\'eduit que la cat\'egorie $\scB$ est de Krull--Schmidt ; en particulier, tout objet admet une d\'ecomposition essentiellement unique en somme d'objets ind\'ecomposables, et un objet est ind\'ecomposable si et seulement si son anneau d'endomorphismes est local.

\begin{rema}
La d\'efinition de $\scB$ donn\'ee ci-dessus n'est pas identique \`a celle adopt\'ee dans~\cite{soergel-bim}. Plus pr\'ecis\'ement, Soergel commence par d\'efinir un morphisme de $\cH_{(W,S)}$ vers l'anneau de Grothendieck scind\'e des $R$-bimodules gradu\'es de type fini \`a gauche et \`a droite, puis d\'efinit $\scB$ comme la sous-cat\'egorie pleine dont les objets sont ceux dont la classe appartient \`a l'image de ce morphisme ; voir~\cite[Definition 5.11]{soergel-bim}. Les objets de $\scB$ sont alors des facteurs directs de sommes de d\'ecal\'es de bimodules de Bott--Samelson (voir~\cite[Lemma 5.13]{soergel-bim}), mais a priori $\scB$ ne contient pas \emph{tous} les tels facteurs directs. Il d\'emontre cependant \`a la fin de son article que $\scB$ est stable par facteurs directs, voir~\cite[Satz 6.14]{soergel-bim}. Cette cat\'egorie contient donc tous les facteurs directs de sommes de d\'ecal\'es de bimodules de Bott--Samelson, ce qui montre qu'elle peut \^etre d\'efinie comme on l'a fait ci-dessus.
\end{rema}

%--------------------------------------------------------------
\subsection{Th\'eor\`eme de cat\'egorification}
%--------------------------------------------------------------

On continue avec les conventions du~\S\ref{ss:bimod-soergel}.
Notons $[\scB]$ le groupe de Grothendieck scind\'e\footnote{Rappelons que le groupe de Grothendieck scind\'e d'une cat\'egorie additive (essentiellement petite) est le quotient du groupe ab\'elien libre engendr\'e par les classes d'isomorphisme d'objets par les relations $[M] = [M'] + [M'']$ si $M \cong M' \oplus M''$.} de $\scB$, qu'on consid\`ere comme une alg\`ebre pour le produit induit par la structure mono\"idale sur $\scB$. Cette alg\`ebre est en fait une $\Z[\vv,\vv^{-1}]$-alg\`ebre pour l'action d\'efinie par $\vv \cdot [M] = [M(1)]$. Le th\'eor\`eme suivant est l'un des r\'esultats principaux de~\cite{soergel-bim}.

\begin{theo}
\label{theo:categorification}
Il existe un unique isomorphisme de $\Z[\vv,\vv^{-1}]$-alg\`ebres
\[
\cH_{(W,S)} \xrightarrow{\sim} [\scB]
\]
envoyant $\uH_s = H_s + \vv$ sur $\B_s$.
\end{theo}

L'unicit\'e dans ce th\'eor\`eme est \'evidente, puisque les \'el\'ements $\uH_s$ engendrent $\cH_{(W,S)}$ comme $\Z[\vv,\vv^{-1}]$-alg\`ebre. Ce qu'il faut d\'emontrer est donc que les classes $\{[\B_s] : s \in S\}$ v\'erifient les relations d\'efinissant $\cH_{(W,S)}$ (ce qui se ram\`ene au cas des groupes dih\'edraux, et peut se faire ``\`a la main'' ; voir~\cite[\S 4]{soergel-bim}), puis que le morphisme obtenu est un isomorphisme. Pour cela, Soergel construit un inverse \`a gauche de ce morphisme, puis classifie les objets ind\'ecomposables de $\scB$, de la fa{\c c}on suivante (voir~\cite[Satz~6.14]{soergel-bim}).

\begin{theo}
\label{theo:classification}
Pour tout $w \in W$ et toute expression r\'eduite $\uw$ de $w$, il existe un unique facteur direct $\B_w$ de $\BS(\uw)$ qui n'est facteur direct d'aucun bimodule gradu\'e de la forme $\BS(\uw')(n)$ avec $n \in \Z$ et $\uw'$ un mot de longueur strictement plus petite que celle de $\uw$. De plus, $\B_w$ ne d\'epend que de $w$, et les bimodules $\{\B_w(n) : w \in W, \, n \in \Z\}$ forment une famille de repr\'esentants des classes d'isomorphisme d'objets ind\'ecomposables dans $\scB$.
\end{theo}

\begin{rema}
\begin{enumerate}
 \item 
Pour $s \in S$, le bimodule gradu\'e $\B_s$ est clairement ind\'ecomposable, donc cette notation n'est pas ambigu\"e : cet objet est bien le bimodule de Soergel ind\'ecomposable associ\'e \`a l'\'el\'ement $s \in W$.
\item
Puisque $R$ est concentr\'e en degr\'es pairs, tout bimodule gradu\'e ind\'ecomposable est soit concentr\'e en degr\'es pairs, soit concentr\'e en degr\'es impairs.
\item
L'objet $\B_w$ n'est d\'efini qu'\`a isomorphisme pr\`es. Dans les cas o\`u la conjecture de Soergel (voir le~\S\ref{ss:conj-Soergel} ci-dessous) est connue, la formule~\eqref{eqn:Hom-fomula} ci-dessous montre que $\End_{\scB}(\B_w)=k$ ; l'isomorphisme consid\'er\'e ci-dessus est donc unique \`a scalaire pr\`es dans ces cas.
\end{enumerate}
\end{rema}

Les propri\'et\'es suivantes sont faciles \`a v\'erifier et extr\`emement utiles. (Ici, le cas $w=s$ de~\eqref{it:ws-1} est la ``cat\'egorification'' des relations quadratiques dans $\cH_{(W,S)}$.)

\begin{lemm}
\label{lemm:BwBs}
Soient $w \in W$ et $s \in S$.
\begin{enumerate}
\item
\label{it:ws-1}
Si $ws<w$, alors $\B_w \B_s \cong \B_w(1) \oplus \B_w(-1)$.
\item
Si $ws>w$, alors $\B_{ws}$ est un facteur direct de $\B_w\B_s$ avec multiplicit\'e $1$, et tous les autres facteurs directs sont de la forme $\B_y(i)$ avec $y < ws$.
\end{enumerate}
\end{lemm}

%Dans la Partie~\ref{sec:global} on utilisera aussi le fait suivant, dont la preuve est \'egalement facile :
%\begin{equation}
%\text{pour tous $w \in W$ et $n \in \Z_{\geq 0}$ on a $\dim\bigl( (\B_w \otimes_R k)^{-n}) = \dim\bigl( (\B_w \otimes_R k)^{n})$.}
%\end{equation}

La preuve du Th\'eor\`eme~\ref{theo:classification} repose sur une formule pour la dimension des espaces de morphismes entre bimodules de Soergel, qui jouera \'egalement un r\^ole important dans la suite. Remarquons que tout $R$-bimodule peut \^etre vu naturellement comme un faisceau coh\'erent sur $V \times V$. Pour tout $x \in W$ on pose
\[
\mathrm{Gr}(x):=\{(x \cdot v, v) : v \in V\} \subset V \times V,
\]
et pour $A \subset W$ fini on note
\[
\mathrm{Gr}(A) = \bigcup_{x \in A} \mathrm{Gr}(x).
\]
Pour tout $R$-bimodule $M$ et tout $A \subset W$, on note $\Gamma_A(M)$ le sous-bimodule de $M$ constitu\'e des \'el\'ements dont le support est contenu dans $\mathrm{Gr}(A')$ (c'est-\`a-dire dont la restriction au compl\'ementaire de $\mathrm{Gr}(A')$ est nulle) pour $A'$ une partie finie de $A$.
%, c'est-\`a-dire des \'el\'ements $m \in M$ tels que l'application $(x,y) \mapsto xmy$ se factorise par $\mathcal{O}(\mathrm{Gr}(A))$. 
Pour $i \in \Z$
%$x \in W$ 
on notera $\Gamma_{\leq i}$ au lieu de $\Gamma_{\{y \in W \mid \ell(y) \leq i\}}$, et de m\^eme pour $\Gamma_{<i}$, $\Gamma_{\geq i}$ et $\Gamma_{>i}$.

Pour $x \in W$ on pose
\[
\Delta_x := \cO(\mathrm{Gr}(x)) (-\ell(x)), \quad \nabla_x := \cO(\mathrm{Gr}(x)) (\ell(x)).
\]
(En termes plus concrets, on peut identifier $\mathrm{Gr}(x)$ \`a $V$ via la premi\`ere projection ; ceci fournit un isomorphisme entre $\Delta_x$ -- ou $\nabla_x$ -- et $R$, sous lequel l'action de $R$ \`a gauche s'identifie \`a l'action naturelle par multiplication. L'action \`a droite d'un \'el\'ement $f \in R$ s'identifie elle \`a la multiplication par $x \cdot f$.) 
Si $M$ est un $R$-bimodule gradu\'e, on dira que $M$ \emph{admet une $\Delta$-filtration}, resp.~\emph{une $\nabla$-filtration}, si $M=\Gamma_A(M)$ pour une partie finie $A \subset W$ et si pour tout $i \in \Z$ il existe un isomorphisme
\begin{multline*}
\Gamma_{\geq i} (M) / \Gamma_{>i} (M) \cong \bigoplus_{\substack{\ell(x)=i \\ n \in \Z}} \bigl( \Delta_x(n) \bigr)^{\oplus d(M,x,n)}, \\
\text{resp.} \quad \Gamma_{\leq i} (M) / \Gamma_{<i} (M) \cong \bigoplus_{\substack{\ell(x)=i \\ n \in \Z}} \bigl( \nabla_x(n) \bigr)^{\oplus d'(M,x,n)}
\end{multline*}
pour des entiers $d(M,x,n) \in \Z_{\geq 0}$, resp.~$d'(M,x,n) \in \Z_{\geq 0}$. On pose alors
\[
\mathrm{ch}_\Delta(M) = \sum_{\substack{x \in W \\ i \in \Z}} d(M,x,i) \vv^i H_x, \quad \text{resp.} \quad \mathrm{ch}_\nabla(M) = \sum_{\substack{x \in W \\ i \in \Z}} d'(M,x,i) \vv^{-i} H_x.
\]

\begin{rema}
\label{rema:Delta-nabla-filtration}
\begin{enumerate}
\item
Il est clair que si $M$ admet une $\Delta$-filtration, respectivement une $\nabla$-filtration, il en est de m\^eme de $M(1)$, et qu'on a $\mathrm{ch}_\Delta(M(1))=\vv \cdot \mathrm{ch}_\Delta(M)$, respectivement $\mathrm{ch}_\nabla(M(1)) = \vv^{-1} \cdot \mathrm{ch}_\nabla(M)$.
\item
\label{it:filtration-libre}
Puisque les bimodules $\Delta_x$ et $\nabla_x$ sont libres comme $R$-modules \`a gauche et \`a droite, tout bimodule gradu\'e qui admet une $\Delta$-filtration ou une $\nabla$-filtration est libre comme $R$-module \`a gauche et \`a droite.
\item
\label{it:Hom-R}
Si $M$ admet une $\nabla$-filtration, alors il existe un isomorphisme canonique
\[
\Hom_{(R,R)}(R, M) \simto \Gamma_{\leq 0}(M)
\]
(o\`u $\Hom_{(R,R)}(-,-)$ d\'esigne l'espace des morphismes de $R$-bimodules non n\'ecessairement gradu\'es). En effet, il est clair que tout morphisme de $R$-bimodules $R \to M$ se factorise par $\Gamma_{\leq 0}(M)$, et comme $\Gamma_{\leq 0}(M)$ est une somme de copies de $\nabla_e$ on a $\Hom_{(R,R)}(R, \Gamma_{\leq 0}(M)) \simto \Gamma_{\leq 0}(M)$. Donc pour conclure il suffit de remarquer que $\Hom_{(R,R)}(R, \Gamma_{\leq i} (M) / \Gamma_{<i} (M))=0$ pour tout $i>0$ et d'utiliser les suites exactes appropri\'ees.
\item
\label{it:ch-Bs}
Choisissons une forme lin\'eaire $\alpha_s$ dont le noyau est $V^s$.
La suite exacte
\[
R(-1) \hookrightarrow \B_s \twoheadrightarrow \mathcal{O}(\mathrm{Gr}(s))(1), \quad \text{resp.} \quad \mathcal{O}(\mathrm{Gr}(s))(-1) \hookrightarrow \B_s \twoheadrightarrow R(1),
\]
o\`u les morphismes sont d\'efinis par $r \mapsto r \cdot \frac{1}{2}(\alpha_s \otimes 1 + 1 \otimes \alpha_s)$, resp.~$r \mapsto r \cdot \frac{1}{2}(\alpha_s \otimes 1 - 1 \otimes \alpha_s)$, et $r \otimes r' \mapsto rs(r')$, resp.~$r \otimes r' \mapsto rr'$, montre que $\mathrm{ch}_\nabla(\B_s) = \underline{H}_s$, resp.~$\mathrm{ch}_\Delta(\B_s)=\underline{H}_s$.
\end{enumerate}
\end{rema}

Soergel d\'emontre alors le r\'esultat suivant dans~\cite{soergel-bim}, qui montre que les bimodules de Soergel v\'erifient des propri\'et\'es similaires \`a celles des objets basculants dans une cat\'egorie de plus haut poids (voir par exemple~\cite[\S 7.5]{riche-hab}). Dans cet \'enonc\'e on note $(-,-)_{\cH}$ la forme $\Z[\vv,\vv^{-1}]$-bilin\'eaire sym\'etrique sur $\cH_{(W,S)}$ (\`a valeurs dans $\Z[\vv,\vv^{-1}]$) qui v\'erifie $(H_x, H_y)_{\cH} = \delta_{xy}$.

\begin{prop}
\label{prop:filtrations-soergel}
\begin{enumerate}
\item
\label{it:filtrations-soergel-1}
Tout bimodule de Soergel $B$ admet une $\Delta$-filtration et une $\nabla$-filtration ; de plus on a
\[
\mathrm{ch}_\Delta(B) = d(\mathrm{ch}_\nabla(B)),
\]
o\`u $d$ est l'involution consid\'er\'ee au~\S{\rm \ref{ss:def-Coxeter}}.
\item
\label{it:filtrations-soergel-2}
L'application $[B] \mapsto \mathrm{ch}_\Delta(B)$ est l'inverse de l'isomorphisme du Th\'eor\`eme~{\rm \ref{theo:categorification}}; en particulier c'est un morphisme d'alg\`ebres, et de m\^eme pour $\mathrm{ch}_\nabla$.
\item
\label{it:filtrations-soergel-3}
%Si $B$ admet une $\Delta$-filtration et $B'$ est un bimodule de Soergel, ou si $B$ est un bimodule de Soergel et $B'$ admet une $\nabla$-filtration, alors
Si $B$ et $B'$ sont des bimodules de Soergel, alors on a
\begin{equation}
\label{eqn:Hom-fomula}
\sum_{i \in \Z} \dim_k \Hom_{\scB}(B,B'(i)) \cdot \vv^{i} = (\mathrm{ch}_\Delta(B), \mathrm{ch}_\nabla(B'))_{\cH}.
\end{equation}
\end{enumerate}
\end{prop}

\begin{proof}
Les points~\eqref{it:filtrations-soergel-1} et~\eqref{it:filtrations-soergel-2} sont d\'emontr\'es dans~\cite[Proposition 5.7 et Proposition 5.10]{soergel-bim}. (Les applications $\mathrm{ch}_\Delta$ et $\mathrm{ch}_\nabla$ sont not\'ees $h_\Delta$ et $h_\nabla$ dans~\cite{soergel-bim}. La formule $\mathrm{ch}_\Delta = d \circ \mathrm{ch}_\nabla$ d\'ecoule du fait que ces deux applications sont des inverses \`a gauche du morphisme du Th\'eor\`eme~\ref{theo:categorification} ; puisque ce morphisme est inversible elles doivent donc co\"incider.) Le point~\eqref{it:filtrations-soergel-3} est d\'emontr\'e dans~\cite[Theorem~5.15]{soergel-bim}.
\end{proof}

\begin{rema}
\label{rema:degre-Bw}
\begin{enumerate}
\item
\label{it:degre-Bw}
Puisque $\mathrm{ch}_\nabla$ est un morphisme d'alg\`ebres, de la Remarque~\ref{rema:Delta-nabla-filtration}\eqref{it:ch-Bs} on d\'eduit que pour tout mot $\uw=(s_1, \cdots, s_n)$ en $S$ on a
\begin{equation}
\label{eqn:ch-nabla-BS}
\mathrm{ch}_\nabla(\BS(\uw)) = \uH_{s_1} \cdots \uH_{s_n}.
\end{equation}
En particulier, pour $w \in W$, en utilisant ceci pour $\uw$ une expression r\'eduite de $w$, on voit que $\mathrm{ch}_\nabla(\B_w) \in \bigoplus_{y \leq w} \Z[\vv,\vv^{-1}] \cdot H_y$, et que le coefficient de $H_w$ dans la d\'ecomposition de $\mathrm{ch}_\nabla(\B_w)$ est $1$ ; il s'ensuit que $(\B_w)^{-\ell(w)} \neq 0$. Comme le degr\'e minimal en lequel $\BS(\uw)$ est non nul est $-\ell(w)$, et qu'il est de dimension $1$ en ce degr\'e,
ceci implique que le degr\'e minimal en lequel $\B_w$ est non nul est \'egalement $-\ell(w)$, et que pour 
%toute expression r\'eduite $\uw$ pour $w$ et 
tout choix d'inclusion scind\'ee $\B_w \hookrightarrow \BS(\uw)$, l'inclusion $(\B_w)^{-\ell(w)} \hookrightarrow \BS(\uw)^{-\ell(w)}$ est un isomorphisme (entre $k$-espaces vectoriels de dimension $1$).
\item
\label{it:factorisation}
Choisissons un raffinement $\preceq$ de l'ordre de Bruhat $\leq$ tel que $(W,\preceq)$ est isomorphe \`a $\Z_{\geq 0}$ muni de son ordre naturel. Alors pour tout $y \in W$ et tout $B \in \scB$, l'action de $R \otimes_k R$ sur $\Gamma_{\preceq y}(B)$ se factorise via $\cO(\mathrm{Gr}(\preceq y))$ (o\`u on note ``$\preceq y$'' pour $\{w \in W \mid w \preceq y\}$). En effet, on peut d\'emontrer cette propri\'et\'e par r\'ecurrence de la fa{\c c}on suivante. Supposons la propri\'et\'e connue pour un \'el\'ement $y$, et soit $z$ l'\'el\'ement suivant de $W$ (pour l'ordre $\preceq$). Soit $p_z \in R$ comme dans~\cite[Notation~6.5]{soergel-bim}. Comme $B$ est libre comme $R$-module \`a droite, pour montrer la propri\'et\'e pour $z$ il suffit de montrer que l'action de $R \otimes_k R$ sur tout \'el\'ement de $\Gamma_{\preceq z}(B) p_z = \Gamma_{\preceq z}(B p_z)$ se factorise par $\cO(\mathrm{Gr}(\preceq z))$. Mais d'apr\`es~\cite[Lemma~6.3 et Satz~6.6]{soergel-bim}, les morphismes naturels
\[
\Gamma_{\{z\}}(B) \to \Gamma_{\leq z}(B p_z) / \Gamma_{<z}(B p_z) \to \Gamma_{\preceq z}(B p_z) / \Gamma_{\preceq z}(B p_z)
\]
sont des isomorphismes; on a donc $ \Gamma_{\preceq z}(B p_z) = \Gamma_{\preceq y}(B) p_z \oplus \Gamma_{\{z\}}(B)$. Par hypoth\`ese de r\'ecurrence l'action de $R \otimes_k R$ sur le premier facteur se factorise par $\cO(\mathrm{Gr}(\preceq y))$, donc a fortiori par $\cO(\mathrm{Gr}(\preceq z))$. Quant au deuxi\`eme facteur, en utilisant~\cite[Proposition~5.9]{soergel-bim} on voit que $B \otimes_R \cO(\mathrm{Gr}(z))$ admet une $\nabla$-filtration, donc en particulier que $\Gamma_{\{z\}}(B)=\Gamma_{\leq 0} \bigl( B \otimes_R \cO(\mathrm{Gr}(z)) \bigr)$ est isomorphe \`a une somme de copies de d\'ecal\'es de $\nabla_z$, de sorte que l'action sur ce facteur se factorise \'egalement par $\cO(\mathrm{Gr}(\preceq z))$.
\end{enumerate}
\end{rema}

%--------------------------------------------------------------
\subsection{Conjecture de Soergel}
\label{ss:conj-Soergel}
%--------------------------------------------------------------

On continue avec les conventions du~\S\ref{ss:bimod-soergel}, en supposant de plus que $k$ est de caract\'eristique nulle. 
En utilisant~\eqref{eqn:ch-nabla-BS} (et l'autodualit\'e des bimodules $\BS(\uw)$ et $\B_w$ qu'on verra au~\S\ref{ss:formes-inv}), on peut v\'erifier (par r\'ecurrence sur $\ell(w)$) que les \'el\'ements $\mathrm{ch}_\nabla(\B_w)$ sont fixes par l'involution $d$ ; voir~\cite[Bemerkung~6.16]{soergel-bim}.
La conjecture suivante, propos\'ee par Soergel dans~\cite{soergel-bim}, donne une description compl\`ete de ces \'el\'ements.

\begin{conj}[``Conjecture de Soergel'']
\label{conj:soergel}
Pour tout $w \in W$, l'image inverse de $[\B_w]$ par l'isomorphisme du Th\'eor\`eme~{\rm \ref{theo:categorification}} est $\uH_w$; en d'autres termes on a
\[
\mathrm{ch}_\Delta(\B_w) = \uH_w.
\]
\end{conj}

Dans le cas consid\'er\'e au Point~\eqref{it:ref-fidele-cartan} de la Proposition~\ref{prop:exemples-ref-fidele}, cette conjecture peut se 
d\'emontrer en utilisant la g\'eom\'etrie de la vari\'et\'e de drapeaux du groupe de Kac--Moody associ\'e; voir le~\S\ref{ss:crystallog} ci-dessous. Elle a \'egalement \'et\'e d\'emontr\'ee alg\'ebriquement dans le cas $|S|=2$ par Soergel (voir~\cite{soergel-bim}) et dans le cas o\`u $m_{st}=\infty$ pour tous $s,t \in S$ distincts par Fiebig (voir~\cite{fiebig}) puis Libedinsky (dans un travail non publi\'e). Le cas g\'en\'eral (pour une repr\'esentation r\'eelle satisfaisant une condition technique) a \'et\'e obtenu par Elias--Williamson~\cite{ew1} ; leur preuve est l'objet principal de cet expos\'e.
%d\'eduire des r\'esultats \'enonc\'es au~\S\ref{ss:crystallog} ci-dessous.

Dans la Partie~\ref{sec:global} on utilisera de fa{\c c}on cruciale l'observation suivante. Si $x,y \in W$ et si la conjecture de Soergel est connue pour $x$ et $y$ (c'est-\`a-dire si $\mathrm{ch}_\Delta(\B_x) = \uH_x$ et $\mathrm{ch}_\Delta(\B_y) = \uH_y$), alors la formule~\eqref{eqn:Hom-fomula} implique que
\begin{equation}
\label{eqn:Hom-BxBy}
\Hom_{\scB}(\B_x, \B_y) = \begin{cases}
k & \text{si $x=y$;} \\
0 & \text{sinon}
\end{cases}
\quad \text{et que} \quad
\Hom_{\scB}(\B_x, \B_y(n))=0 \text{ si $n<0$.}
\end{equation}
%Ensuite, si $w \in W$ et $s \in S$ sont tels que $ws>w$, alors du fait que $\uH_w \uH_s \in \bigoplus_{y} \Z \cdot \uH_y$ on d\'eduit que si la conjecture de Soergel est connue pour les \'el\'ements de longueur $\leq \ell(w)$, alors
%\begin{equation}
%\text{$\B_w\B_s$ est une somme d'objets de la forme $\B_y$ (sans d\'ecalage).}
%\end{equation}

%--------------------------------------------------------------
\subsection{Le cas des syst\`emes de Coxeter cristallographiques}
\label{ss:crystallog}
%--------------------------------------------------------------

Soit $A$ une matrice de Cartan g\'en\'eralis\'ee sym\'etrisable, soit $(W,S)$ le syst\`eme de Coxeter associ\'e, et soit $\frh$ une r\'ealisation de $A$ sur $\R$ au sens de Kac (voir le~\S\ref{ss:ref-fidele}). D'apr\`es la Proposition~\ref{prop:exemples-ref-fidele}\eqref{it:ref-fidele-cartan}, la repr\'esentation de $W$ sur $\frh^*$ est r\'eflexion fid\`ele. Puisque $A$ est sym\'etrisable il existe un isomorphisme $W$-\'equivariant $\frh \cong \frh^*$ (voir~\cite[Proposition~1.5.2]{kumar}), de sorte que la repr\'esentation de $W$ sur $\frh$ est \'egalement r\'eflexion fid\`ele. On notera $\scB$ la cat\'egorie des bimodules de Soergel associ\'ee \`a $(W,S)$ et \`a son action sur $\frh$.
%$\frh_\C:=\C \otimes_\R \frh$.

Dans ce cadre, cette cat\'egorie peut se d\'ecrire g\'eom\'etriquement de la fa{\c c}on suivante. \`A $A$ on peut associer un groupe de Kac--Moody $\bG$ avec un sous-groupe de Borel $\bB$. (Ici on consid\`ere le groupe de Kac--Moody ``minimal'' construit dans~\cite[\S 7.4]{kumar}, de sorte que $\bG$ est un ind-sch\'ema en groupes affine, cf.~\cite[Theorem~7.4.14]{kumar}.) On consid\`ere la cat\'egorie d\'eriv\'ee $\bB$-\'equivariante $\Db_{\bB}(\bG/\bB, \R)$ \`a coefficients dans $\R$ au sens de Bernstein--Lunts~\cite{bl} ; cette cat\'egorie admet un bifoncteur de \emph{convolution} qui en fait une cat\'egorie mono\"idale. On notera $\widetilde{\scB}$ la sous-cat\'egorie pleine des complexes semi-simples (c'est-\`a-dire des sommes directes de d\'ecal\'es cohomologiques de complexes de cohomologie d'intersection de $\bB$-orbites) dans $\Db_{\bB}(\bG/\bB, \R)$. Le th\'eor\`eme de d\'ecomposition (dans sa version \'equivariante) assure que cette sous-cat\'egorie est stable par convolution. On notera $[\widetilde{\scB}]$ le groupe de Grothendieck scind\'e de $\widetilde{\scB}$.

L'isomorphisme de Borel fournit un isomorphisme d'alg\`ebres gradu\'ees
\[
R = \mathrm{Sym}(\frh^*) \simto \mathsf{H}^\bullet_{\bB}(\mathrm{pt}; \R).
\]
En utilisant cette identification, pour tout $\mathcal{F}$ dans $\Db_{\bB}(\bG/\bB, \R)$, l'espace vectoriel gradu\'e $\mathsf{H}^\bullet_{\bB}(\bG/\bB, \mathcal{F})$ admet une structure naturelle de $R$-bimodule gradu\'e. On obtient ainsi un foncteur de $\Db_{\bB}(\bG/\bB, \R)$ vers la cat\'egorie des $R$-bimodules gradu\'es, dont on notera $\mathbb{H}$ la restriction \`a $\widetilde{\scB}$.

Le r\'esultat suivant est d\'emontr\'e par Soergel dans~\cite{soergel-philosophy} dans le cas o\`u $A$ est de type fini (c'est-\`a-dire quand $\bG$ est un groupe alg\'ebrique semisimple). Le cas g\'en\'eral est trait\'e dans~\cite{haerterich} et~\cite{by}.

\begin{theo}
\label{theo:H-equiv}
Le foncteur $\mathbb{H}$ induit une \'equivalence de cat\'egories mono\"idales $\widetilde{\scB} \simto \scB$.
\end{theo}

Via l'\'equivalence du Th\'eor\`eme~\ref{theo:H-equiv}, l'application $\mathrm{ch}_\Delta$ s'identifie au morphisme
\[
\mathrm{ch}_\bG : \left\{
\begin{array}{ccc}
[\widetilde{\scB}] & \to & \cH_{(W,S)} \\[2pt]
[\mathcal{F}]
& \mapsto & \sum_{\substack{i \in \Z \\ y \in W}} \dim_{\R} \mathsf{H}^i(\mathcal{F}_{y\bB}) \vv^{-i-\ell(y)} H_y
\end{array}
\right. .
\]
Dans ce cas, la Conjecture~\ref{conj:soergel} d\'ecoule donc du fait que
\[
\mathrm{ch}_{\bG}(\mathcal{IC}(\bB w \bB/\bB, \underline{\R})) = \underline{H}_w
\]
(un r\'esultat d\^u \`a Kazhdan--Lusztig ; voir \'egalement~\cite{springer} pour une preuve plus simple).

%--------------------------------------------------------------
\subsection{Modules de Soergel}
\label{ss:modules-Soergel}
%--------------------------------------------------------------

Le r\'esultat technique d\'emontr\'e dans ce paragraphe sera utilis\'e au~\S\ref{ss:preuve-KL}, mais n'est pas utile \`a la compr\'ehension de la Partie~\ref{sec:global}.

On suppose que $W$ est fini. Si $M,N$ sont des $R$-bimodules, resp.~des $R$-modules, on notera $\Hom_{(R,R)}(M,N)$, resp.~$\Hom_R(M,N)$, le $R$-bimodule, resp.~$R$-module, des morphismes de $R$-bimodules, resp.~$R$-modules, non n\'ecessairement gradu\'es de $M$ vers $N$. Si $M$ et $N$ sont des bimodules de Soergel, le morphisme naturel
\[
\bigoplus_{n \in \Z} \Hom_{\scB}(M,N(n)) \to \Hom_{(R,R)}(M,N)
\]
est un isomorphisme. Un \'enonc\'e similaire est vrai \'egalement pour les $R$-modules gradu\'es de type fini.

Dans le cas particulier des groupes de Weyl des groupes alg\'ebriques r\'eductifs, l'\'enonc\'e suivant est un cas particulier de~\cite[Proposition~8]{soergel-HC}. (La preuve de~\cite{soergel-HC} repose sur la th\'eorie des repr\'esentations et un argument de ``d\'eformation''. Alternativement, on peut \'egalement d\'emontrer ce r\'esultat dans le cas des groupes de Weyl en utilisant la g\'eom\'etrie de la vari\'et\'e de drapeaux associ\'ee.) La preuve alg\'ebrique donn\'ee ci-dessous (dont l'existence est mentionn\'ee dans~\cite[Remark~3.2]{williamson-torsion}) 
%m'a \'et\'e aimablement communiqu\'ee par Wolfgang 
est due \`a Soergel.

\begin{prop}
\label{prop:modules-Soergel}
Pour tous bimodules de Soergel $B,B'$, le morphisme naturel
\[
\Hom_{(R,R)}(B,B') \otimes_R k \to \Hom_R(B \otimes_R k, B' \otimes_R k)
\]
est un isomorphisme.
\end{prop}

\begin{proof}
%\footnote{A REVOIR !!!}
%\footnote{Cette preuve m'a \'et\'e aimablement communiqu\'ee par Wolfgang Soergel.}
En utilisant~\cite[Proposition~5.10]{soergel-bim} on se ram\`ene au case $B=R$, et on note alors $B$ pour $B'$. Rappelons que $B$ poss\`ede une $\nabla$-filtration ; voir la Proposition~\ref{prop:filtrations-soergel}. En utilisant la Remarque~\ref{rema:Delta-nabla-filtration}\eqref{it:Hom-R} dans le premier cas, on a des isomorphismes canoniques
\[
\Hom_{(R,R)}(R,B) \cong \Gamma_{\leq 0}(B), \quad \Hom_R(k, B \otimes_R k) \cong \{v \in B \otimes_R k \mid R^{>0} \cdot v=0\},
\]
o\`u $R^{>0} := \bigoplus_{i>0} R^i$.
L'inclusion $\Gamma_{\leq 0}(B) \hookrightarrow B$ est scind\'ee comme morphisme de $R$-modules \`a droite (voir la Remarque~\ref{rema:Delta-nabla-filtration}\eqref{it:filtration-libre}), ce qui implique que notre morphisme est injectif. Pour conclure, il faut donc montrer que tout \'el\'ement de $B \otimes_R k$ annul\'e par $R^{>0}$ appartient au sous-espace $\Gamma_{\leq 0}(B) \otimes_R k \subset B \otimes_R k$.

Soit donc $b \in B \otimes_R k$ non nul, et soit $i \in \Z_{\geq 0}$ l'unique entier tel que $b$ appartient \`a $\Gamma_{\leq i}(B) \otimes_R k$ mais pas \`a $\Gamma_{\leq i-1}(B) \otimes_R k$. On va montrer que si $i \neq 0$, alors $b$ n'est pas annul\'e par $R^{>0}$. Pour cela, choisissons une pr\'eimage $a$ de $b$ dans $\Gamma_{\leq i}(B)$. \'Ecrivons le quotient $M=\Gamma_{\leq i}(B) / \Gamma_{\leq i-1}(B)$ comme une somme directe
\[
M= \bigoplus_{\substack{y \in W \\ \ell(y)=i}} M_y
\]
o\`u chaque $M_y$ est une somme de copies de d\'ecal\'es de $\nabla_y$, et consid\'erons la d\'ecomposition
\[
\overline{a} = \sum_{y} \overline{a}_y
\]
de l'image $\overline{a}$ de $a$ dans $M$ selon ces facteurs directs. Par hypoth\`ese l'image de $\overline{a}$ dans $M \otimes_R k$ est non nulle, donc il existe $y$ tel que l'image de $\overline{a}_y$ dans $M \otimes_R k$ est non nulle.

Soit maintenant $w_0 \in W$ l'\'el\'ement de plus grande longueur, et
consid\'erons un \'el\'ement $c'_{w_0}$ de $(R \otimes_{R^W} R)^{2\ell(w_0)}$ non nul et qui s'annule sur chaque $\mathrm{Gr}(x)$ avec $x \neq w_0$, voir~\cite[Proposition~2]{soergel-HC}. Choisissons \'egalement une pr\'eimage $c_{w_0}$ de $c'_{w_0}$ dans $R \otimes_k R$. Alors, si les op\'erateurs de Demazure $(\partial_x)_{x \in W}$ sont d\'efinis comme dans~\cite[\S 2.2]{soergel-HC}, l'\'el\'ement $c_y := \partial_{y w_0}(c_{w_0}) \in (R \otimes_k R)^{2\ell(y)}$ s'annule sur tous les $\mathrm{Gr}(z)$ avec $z \not \geq y$ d'apr\`es~\cite[Lemma~5]{soergel-HC}.
En particulier, au vu de la Remarque~\ref{rema:degre-Bw}\eqref{it:factorisation}, le produit $c_y \cdot a$ ne d\'epend que de $\overline{a}_y$.
%Si on note $\overline{c}_y$ l'image de $c_y$ dans $R = (R \otimes_k R) \otimes_R k$,
%la Remarque~\ref{rema:degre-Bw}\eqref{it:factorisation} montre alors que 
%$c_y \cdot \overline{a}_z = 0$ si $z \neq y$. 
D'autre part, il est remarqu\'e dans~\cite[Bemerkung~6.7]{soergel-bim} que la multiplication par $c_y$ induit un isomorphisme (non gradu\'e) $\Gamma_{\leq y}(B) / \Gamma_{< y}(B) \simto \Gamma_{\{y\}}(B)$. Donc l'image de $c_y \cdot a$ dans $\Gamma_{\{y\}}(B) \otimes_R k$ est non nulle. Finalement, comme not\'e dans la Remarque~\ref{rema:degre-Bw}\eqref{it:factorisation}, le bimodule $B \otimes_R \cO(\mathrm{Gr}(y))$ admet une $\nabla$-filtration, de sorte que l'inclusion $\Gamma_{\{y\}}(B) \hookrightarrow B$ est scind\'ee comme morphisme de $R$-modules \`a droite, et donc le morphisme naturel $\Gamma_{\{y\}}(B) \otimes_R k \to B \otimes_R k$ est injectif. On a finalement montr\'e que l'image de $c_y \cdot a$ dans $B \otimes_R k$ est non nulle, ce qui ach\`eve la preuve puisque cette image coincide avec $\overline{c}_y \cdot b$, o\`u $\overline{c}_y$ est l'image de $c_y$ dans $R = (R \otimes_k R) \otimes_R k$, et donc un \'el\'ement $R^{>0}$ (puisque $i=\ell(y)$ est non nul par hypoth\`ese).
\end{proof}

Il d\'ecoule en particulier de la Proposition~\ref{prop:modules-Soergel} que si $w \in W$ l'anneau des endomorphismes de $\B_w \otimes_R k$ comme $R$-module gradu\'e est un quotient de $\End_{\scB}(\B_w)$, donc est un anneau local. Donc ce module gradu\'e est ind\'ecomposable. En appliquant des r\'esultats g\'en\'eraux sur les alg\`ebres de dimension finie (voir~\cite{gg}), on obtient m\^eme que $\B_w \otimes_R k$ est ind\'ecomposable comme $R$-module \emph{non gradu\'e}. En particulier, il s'ensuit que les modules $\B_w \otimes_R k$ sont exactement les $R$-modules (non gradu\'es) qu'on peut obtenir comme facteurs directs ind\'ecomposables des modules de la forme
\[
\BS(\uw) \otimes_R k = R \otimes_{R^{s_1}} R \otimes_{R^{s_2}} \cdots \otimes_{R^{s_n}} k
\]
o\`u $\uw=(s_1, \cdots, s_n)$ est un mot en $S$.

\begin{rema}
 La preuve ci-dessus utilise de fa{\c c}on cruciale l'hypoth\`ese que $W$ est fini. En fait, la Proposition~\ref{prop:modules-Soergel} est fausse en g\'en\'eral si $W$ est infini. Des exemples explicites de bimodules de Soergel $B$ ind\'ecomposables tels que $B \otimes_R k$ est d\'ecomposable ont \'et\'e exhib\'es par L.~Patimo.
\end{rema}

%%%%%%%%%%%%%%%%%%%%%%%%%%%%%%%
\section{La th\'eorie de Hodge des bimodules de Soergel : cas global}
\label{sec:global}
%%%%%%%%%%%%%%%%%%%%%%%%%%%%%%%

Dans cette partie, le corps de base de tous les espaces vectoriels consid\'er\'es sera (tacitement) $\R$. 

%--------------------------------------------------------------
\subsection{Hypoth\`eses}
%--------------------------------------------------------------

On fixe un syst\`eme de Coxeter $(W,S)$ et une repr\'esentation $V$ de $W$ qui est r\'eflexion fid\`ele. On fixe \'egalement des \'el\'ements $(\alpha_s^\vee)_{s \in S}$ de $V$ et $(\alpha_s)_{s \in S}$ de $V^*$ tels que $s \cdot v = v-\langle \alpha_s, v \rangle \alpha_s^\vee$ pour tous $s \in S$ et $v \in V$. (De tels \'el\'ements existent puisque chaque $s \in S$ agit par une r\'eflexion sur $V$.) On suppose qu'il existe un \'el\'ement $\rho \in V^*$ tel que pour $s \in S$ et $w \in W$ on a
\begin{equation}
\label{eqn:rho-positivite}
\langle w(\rho), \alpha_s^\vee \rangle > 0 \quad \Leftrightarrow \quad sw>w.
\end{equation}
(En particulier, si $w=1$, ceci implique que $\langle \rho, \alpha_s^\vee \rangle >0$ pour tout $s \in S$.) 

Dans les situations consid\'er\'ees dans la Proposition~\ref{prop:exemples-ref-fidele}, dans le cas~\eqref{it:ref-fidele-soergel} on peut poser $\alpha_s^\vee = e_s$, $\alpha_s = e_s^\vee$, et prendre pour $\rho$ tout \'el\'ement tel que $\langle \rho, e_s \rangle >0$ pour tout $s \in S$ (voir~\cite[\S V.4.4]{bourbaki}, en remarquant que $\mathrm{Vect}(\{e_s : s \in S\})$ est isomorphe comme $W$-repr\'esentation \`a la repr\'esentation g\'eom\'etrique de $W$) ; dans le cas~\eqref{it:ref-fidele-cartan} on peut prendre pour les $\alpha_s$, resp.~$\alpha_s^\vee$, les racines simples, resp.~coracines simples, et pour $\rho$ tout \'el\'ement tel que $\langle \rho, \alpha_s^\vee \rangle >0$ pour tout $s \in S$ (voir~\cite[Lemma~1.3.13]{kumar}).

%--------------------------------------------------------------
\subsection{Alg\`ebre lin\'eaire ``de Hodge''}
%--------------------------------------------------------------

Dans cette sous-partie on fixe une terminologie qui permet d'\'enoncer de fa{\c c}on ``compacte'' les r\'esultats principaux de~\cite{ew1}. Cette terminologie est motiv\'ee par la th\'eorie de Hodge des vari\'et\'es projectives complexes ; voir~\cite[\S\S 2--3]{williamson-survey} pour plus de d\'etails. (Notons toutefois que les d\'efinitions consid\'er\'ees ici sont l\'eg\`erement diff\'erentes de celles consid\'er\'ees dans~\cite{williamson-survey}.)

On appellera \emph{donn\'ee de Lefschetz} une paire $(V,L)$ o\`u $V$ est un espace vectoriel gradu\'e de dimension finie concentr\'e soit en degr\'es pairs, soit en degr\'es impairs, et $L : V \to V(2)$ est une application lin\'eaire gradu\'ee. On dira que $(V,L)$ \emph{v\'erifie le th\'eor\`eme de Lefschetz difficile} si, pour tout $i \geq 0$, l'application lin\'eaire $L^{\circ i} : V^{-i} \to V^i$ est un isomorphisme. Dans ce cas, en posant, pour tout $i \geq 0$,
\[
P_L^{-i} := \{v \in V^{-i} \mid L^{\circ (i+1)}(v)=0\},
\]
on obtient un isomorphisme canonique de $\R[L]$-modules gradu\'es
\[
\bigoplus_{i \geq 0} \R[L]/L^{i+1} \otimes_\R P_L^{-i} \xrightarrow{\sim} V.
\]
(Le sous-espace $P_L^{-i} \subset V^{-i}$ est appel\'e sous-espace des \emph{\'el\'ements primitifs} de degr\'e $-i$.)

On appellera \emph{donn\'ee de Hodge--Riemann} un triplet $(V,L,\langle -,- \rangle)$ o\`u $(V,L)$ est une donn\'ee de Lefschetz et $\langle -,- \rangle : V \times V \to \R$ est une forme bilin\'eaire sym\'etrique non d\'eg\'en\'er\'ee gradu\'ee (c'est-\`a-dire telle que $\langle V^i, V^j \rangle =0$ si $i+j \neq 0$) pour laquelle $L$ est autoadjoint (c'est-\`a-dire v\'erifie $\langle Lx, y \rangle = \langle x, Ly \rangle$ pour tous $x,y \in V$). On dira que $(V,L,\langle -,- \rangle)$ \emph{v\'erifie les relations de Hodge--Riemann}\footnote{Dans~\cite{ew1}, les auteurs autorisent les signes ``oppos\'es'' \`a ceux consid\'er\'es ici dans les relations de Hodge--Riemann. Ils parlent alors de ``signes standard'' pour les conditions consid\'er\'ees dans ces notes.} si, en notant $\min$ le degr\'e minimal en lequel $V$ est non nul, pour tout $i \in 2\Z_{\geq 0}$ tel que $\min+i\leq 0$, la restriction de la forme bilin\'eaire sym\'etrique sur $V^{\min+i}$ d\'efinie par
\begin{equation*}
%\label{eqn:forme-Lefschetz}
(x,y) \mapsto \langle x,L^{\circ (-\min-i)} (y) \rangle
\end{equation*}
\`a $P^{\min+i}_L =\{v \in V^{\min+i} \mid L^{\circ (-\min-i+1)}(v)=0\}$ est $(-1)^{\frac{i}{2}}$-d\'efinie.

Les relations de Hodge--Riemann sont ``plus fortes'' que le th\'eor\`eme de Lefschetz difficile, au sens suivant.

\begin{lemm}
\label{lemm:HR-hL}
Soit $(V,L,\langle -,- \rangle)$ une donn\'ee de Hodge--Riemann qui v\'erifie les relations de Hodge--Riemann, et supposons que $\dim(V^i) = \dim(V^{-i})$ pour tout $i \in \Z_{\geq 0}$. Alors $(V,L)$ v\'erifie le th\'eor\`eme de Lefschetz difficile.
\end{lemm}

\begin{proof}
%D'apr\`es notre hypoth\`ese, il suffit de montrer que $L^i : V^{-i} \to V^i$ est injectif pour tout $i \geq 0$. On montre cette propri\'et\'e par r\'ecurrence descendante sur $i$, le cas initial \'etant $i=-\min$, o\`u $\min$ est le degr\'e minimal en lequel $V$ est non nul. Dans ce cas, l'injectivit\'e
Nous allons d\'emontrer par r\'ecurrence descendante sur $j$ que, pour tout $j \geq 0$, la forme bilin\'eaire sym\'etrique sur  $V^{-j}$ d\'efinie par $(x,y) \mapsto \langle x,L^{\circ j} (y) \rangle$ est non d\'eg\'en\'er\'ee. Cette propri\'et\'e implique que $L^{\circ j} : V^{-j} \to V^j$ est injective, et donc un isomorphisme d'apr\`es nos hypoth\`eses.

Le cas initial est $j=-\min$ (o\`u $\min$ est le degr\'e minimal en lequel $V$ est non nul), qui d\'ecoule des relations de Hodge--Riemann en degr\'e $\min$ puisque $L^{\circ (-\min+1)}=0$. Supposons maintenant le r\'esultat d\'emontr\'e pour $j+2$ (c'est-\`a-dire en degr\'e $-j-2$). Puisque $L^{\circ (j+2)} : V^{-j-2} \to V^{j+2}$ est un isomorphisme, $L$ est injectif sur $V^{-j-2}$ et on a $V^{-j} = L(V^{-j-2}) \oplus P_L^{-j}$. De plus il est facile de voir que cette d\'ecomposition est orthogonale pour la forme sur $V^{-j}$ consid\'er\'ee ci-dessus, et que $L$ induit une isom\'etrie de $V^{-j-2}$ (muni de $(x,y) \mapsto \langle x,L^{\circ (j+2)} (y) \rangle$) vers $L(V^{-j-2})$ (muni de $(x,y) \mapsto \langle x,L^{\circ j} (y) \rangle$). On d\'eduit de notre hypoth\`ese de r\'ecurrence que la restriction de notre forme \`a $L(V^{-j-2})$ est non d\'eg\'en\'er\'ee. Sa restriction \`a $P_L^{-j}$ est \'egalement non d\'eg\'en\'er\'ee par hypoth\`ese. Donc la forme est non d\'eg\'en\'er\'ee sur $V^{-j}$.
\end{proof} 

Dans l'autre sens, la propri\'et\'e de Lefschetz difficile peut parfois aider \`a d\'emontrer les relations de Hodge--Riemann gr\^ace au r\'esultat suivant, dont la preuve utilise les m\^emes m\'ethodes que la preuve pr\'ec\'edente.

\begin{lemm}
Soit $(V,L,\langle -,- \rangle)$ une donn\'ee de Hodge--Riemann, et supposons que $(V,L)$ v\'erifie le th\'eor\`eme de Lefschetz difficile. Alors $(V,L,\langle -,- \rangle)$ v\'erifie les relations de Hodge--Riemann si et seulement si pour tout $i \in 2\Z_{\geq 0}$ tel que $\min+i \leq 0$ (o\`u $\min$ est le degr\'e minimal en lequel $V$ est non nul), la forme bilin\'eaire sur $V^{\min+i}$ d\'efinie par $(x,y) \mapsto \langle x,L^{\circ(-\min-i)} (y) \rangle$ a pour signature\footnote{Ici, la \emph{signature} d'une forme bilin\'eaire sym\'etrique est interpr\'et\'ee comme $r-s$, o\`u $r$, respectivement $s$, est le nombre de coefficients strictement positifs, respectivement strictement n\'egatifs, dans la d\'ecomposition de la forme quadratique associ\'ee en somme de carr\'es de formes lin\'eaires lin\'eairement ind\'ependantes.}
\[
\dim(V^{\min}) - \bigl( \dim(V^{\min+2})-\dim(V^{\min}) \bigr) + \cdots + (-1)^{\frac{i}{2}} \bigl( \dim(V^{\min+i}) - \dim(V^{\min+i-2}) \bigr).
\]
\end{lemm}

En utilisant le fait que la signature d'une famille continue de formes bilin\'eaires sym\'etriques non d\'eg\'en\'er\'ees ne peut pas varier, on en d\'eduit le principe suivant, qui est crucial dans la preuve de~\cite{ew1}.

\begin{coro}
\label{coro:continuite}
Soit $I$ un intervalle de $\R$. On suppose donn\'es:
\begin{itemize}
\item
un espace vectoriel gradu\'e $V$ de dimension finie concentr\'e soit en degr\'es paires, soit en degr\'es impaires;
\item
une forme bilin\'eaire sym\'etrique non d\'eg\'en\'er\'ee gradu\'ee $\langle -,- \rangle : V \times V \to \R$;
\item
une famille continue $(L_t)_{t \in I}$ d'applications lin\'eaires gradu\'ees $V \to V(2)$ autoadjointes pour $\langle -,- \rangle$.
\end{itemize}
Supposons que $(V,L_t)$ v\'erifie le th\'eor\`eme de Lefschetz difficile pour tout $t \in I$. S'il existe $t_0 \in I$ tel que $(V,L_{t_0},\langle -,- \rangle)$ v\'erifie les relations de Hodge--Riemann, alors $(V,L_{t},\langle -,- \rangle)$ v\'erifie les relations de Hodge--Riemann pour tout $t \in I$.
\end{coro}

Enfin, le lemme suivant 
(voir~\cite[Lemma~2.3]{ew1}) 
joue un r\^ole crucial pour d\'emontrer le th\'eor\`eme de Lefschetz difficile.

\begin{lemm}
\label{lemm:Lefschetz-recurrence}
Soit $(V,L_V,\langle -,- \rangle_V)$ une donn\'ee de Hodge--Riemann qui v\'erifie les relations de Hodge--Riemann, et soit $V'$ un espace vectoriel gradu\'e de dimension finie muni d'une forme bilin\'eaire sym\'etrique $\langle -,- \rangle_{V'}$ (qu'on ne suppose pas non d\'eg\'en\'er\'ee) et d'une application lin\'eaire gradu\'ee $L_{V'} : V' \to V'(2)$. Soit \'egalement $\varphi : V' \to V(1)$ une application lin\'eaire gradu\'ee qui est injective en degr\'es $\leq -1$ et qui v\'erifie
\[
\varphi \circ L_{V'} = L_V \circ \varphi \quad \text{et} \quad \langle \varphi(x), \varphi(y) \rangle_{V} = \langle x, L_{V'}(y) \rangle_{V'} \ \text{pour tous $x,y \in V'$.}
\]
Alors $(L_{V'})^{\circ i} : (V')^{-i} \to (V')^i$ est injective pour tout $i \geq 0$.
\end{lemm}

\begin{proof}
Le r\'esultat est \'evident pour $i=0$.
Soient maintenant $i \geq 1$ et $x \in (V')^{-i} \smallsetminus \{0\}$. Si $0 \neq (L_V)^{\circ i}(\varphi(x)) = \varphi \bigl( (L_{V'})^{\circ i}(x) \bigr)$, alors $(L_{V'})^{\circ i}(x) \neq 0$. Sinon, $\varphi(x)$ est un \'el\'ement primitif non nul de $V^{-i+1}$, de sorte que $0 \neq \langle \varphi(x), (L_V)^{\circ (i-1)}(\varphi(x)) \rangle_V = \langle x, (L_{V'})^{\circ i}(x) \rangle_{V'}$, et donc $(L_{V'})^{\circ i}(x) \neq 0$.
\end{proof}

%--------------------------------------------------------------
\subsection{Formes invariantes}
\label{ss:formes-inv}
%--------------------------------------------------------------

Si $B$ est un bimodule de Soergel, une \emph{forme invariante} sur $B$ est une application $\R$-bilin\'eaire gradu\'ee\footnote{Ici, cette condition signifie que $\langle x,y \rangle_B \in R^{i+j}$ si $x \in B^i$ et $y \in B^j$.}
\[
\langle -,- \rangle_B : B \times B \to R
\]
qui v\'erifie $\langle rx,y \rangle = \langle x,ry \rangle$ et $\langle xr,y \rangle = \langle x,yr \rangle = \langle x,y \rangle r$ pour tous $x,y \in B$ et $r \in R$. La donn\'ee d'une telle forme est \'equivalente, via $x \mapsto \langle x,- \rangle_B$ \`a celle d'un morphisme de $R$-bimodules  gradu\'es $B \to \D(B)$, o\`u on a d\'efini $\D$ par
\[
\D(B) := \bigoplus_{i \in \Z} \Hom_{\mathrm{Mod^\Z}-R}(B,R(i)),
\]
o\`u $\mathrm{Mod^\Z}-R$ d\'esigne la cat\'egorie des $R$-modules \`a droite gradu\'es, et o\`u la structure de $R$-bimodule est d\'efinie par $(rfr')(x)=f(rxr')$ pour $r,r' \in R$, $f \in \D(B)$, $x \in B$. On dira que cette forme invariante est \emph{non d\'eg\'en\'er\'ee} si le morphisme associ\'e $B \to \D(B)$ est un isomorphisme.

Le lemme suivant (facile \`a v\'erifier) fournit la m\'ethode principale pour construire de telles formes.

\begin{lemm}
\label{lemm:formes}
Soient $B,B'$ des bimodules de Soergel, et $\langle -,- \rangle_B$, $\langle -,- \rangle_{B'}$ des formes invariantes sur $B$ et $B'$. On d\'efinit une forme $\R$-bilin\'eaire sur $B B'$ en posant
\[
\langle b_1 \otimes b_1', b_2 \otimes b'_2 \rangle_{B B'} =\langle (\langle b_1,b_2 \rangle_B) \cdot b'_1, b'_2 \rangle_{B'}.
\]
Alors $\langle -, - \rangle_{B B'}$ est une forme invariante sur $B B'$, qui est non d\'eg\'en\'er\'ee, resp.~sym\'etrique, si $\langle -, - \rangle_B$ et $\langle -,- \rangle_{B'}$ sont non d\'eg\'en\'er\'ees, resp.~sym\'etriques.
\end{lemm}

En particulier, on peut construire une forme invariante sym\'etrique non d\'eg\'en\'er\'ee sur $\B_s$ de la fa{\c c}on suivante : les vecteurs
\[
c_e := 1 \otimes 1 \quad \text{et} \quad c_s := \frac{1}{2}(\alpha_s \otimes 1 + 1 \otimes \alpha_s)
\]
forment une base de $\B_s$ comme $R$-module \`a droite ; on d\'efinit alors $\langle r_1 \otimes r_2, r'_1 \otimes r'_2 \rangle_{\B_s}$ comme le coefficient de $r_1 r'_1 \otimes r_2 r'_2$ sur $c_s$ dans la base pr\'ec\'edente. Pour tout mot $\uw$ en $S$, la construction du Lemme~\ref{lemm:formes} fournit alors une forme invariante sym\'etrique et non d\'eg\'en\'er\'ee sur $\BS(\uw)$, qu'on notera $\langle -,- \rangle_{\BS(\uw)}$ et qu'on appellera \emph{forme d'intersection}\footnote{Cette terminologie est justifi\'ee par l'analogie avec la preuve du th\'eor\`eme de d\'ecomposition par de Cataldo et Migliorini (voir le~\S\ref{ss:structure-preuve}), o\`u certaines formes d'intersection en homologie de Borel--Moore jouent un r\^ole crucial. Les formes $\langle -,- \rangle_{\BS(\uw)}$ peuvent \^etre consid\'er\'ees comme des analogues alg\'ebriques de ces formes.}.

L'existence de la forme d'intersection sur $\BS(\uw)$ montre en particulier que $\D(\BS(\uw)) \cong \BS(\uw)$. En utilisant la caract\'erisation de $\B_w$ donn\'ee au Th\'eor\`eme~\ref{theo:classification} on en d\'eduit que pour tout $w \in W$ on a
\begin{equation}
\label{eqn:D-Bw}
\D(\B_w) \cong \B_w.
\end{equation}
En particulier, ceci implique que $\dim((\B_w \otimes_R \R)^{-i}) = \dim((\B_w \otimes_R \R)^{i})$ pour tout $i \geq 0$.

De la forme d'intersection sur $\BS(\uw)$ on peut d\'eduire (par restriction) une forme invariante sym\'etrique non d\'eg\'en\'er\'ee sur chaque $\B_w$ gr\^ace au lemme suivant.

\begin{lemm}
\label{lemm:restriction-forme}
Si $B$ est un bimodule de Soergel muni d'une forme invariante non-d\'eg\'en\'er\'ee $\langle -,- \rangle_B$, si $w \in W$, et si $\B_w$ apparait comme facteur direct de $B$ avec multiplicit\'e $1$, alors pour tout choix d'inclusion scind\'ee $\B_w \hookrightarrow B$ la restriction de $\langle -,- \rangle_B$ \`a $\B_w$ est non d\'eg\'en\'er\'ee.
\end{lemm}

\begin{proof}
Notons $f : \B_w \hookrightarrow B$ l'inclusion scind\'ee choisie.
Notre forme sur $B$ induit un isomorphisme $B \xrightarrow{\sim} \D(B)$. 
%Pour tout $x \in W$, par unicit\'e de $\B_x$ on a $\D(\B_x) \cong \B_x$. Donc n
Notre hypoth\`ese et~\eqref{eqn:D-Bw} montrent que $\B_w$ est facteur direct de $\D(B)$ avec multiplicit\'e $1$, et que $\D(f) : \D(B) \twoheadrightarrow \D(\B_w)$ est une projection sur un tel facteur. 
%Si on fixe un scindage $B \to \B_w$ de notre inclusion, 
La preuve de l'unicit\'e de la d\'ecomposition en facteurs ind\'ecomposables dans une cat\'egorie de Krull--Schmidt montre alors que la compos\'ee
\[
\B_w \hookrightarrow B \xrightarrow{\sim} \D(B) \twoheadrightarrow \D(\B_w)
\]
est un isomorphisme, ce qui implique que la restriction consid\'er\'ee est non d\'eg\'en\'er\'ee.
\end{proof}

%Si la Conjecture~\ref{conj:soergel} est v\'erifi\'ee pour $w$, c'est-\`a-dire si $\mathrm{ch}(\B_w)=\uH_w$, alors~\eqref{eqn:Hom-fomula} montre que 
Si la conjecture de Soergel est v\'erifi\'ee pour $w$ (c'est-\`a-dire si $\mathrm{ch}_\Delta(\B_w)=\uH_w$) alors~\eqref{eqn:Hom-BxBy} et~\eqref{eqn:D-Bw} montrent que
\[
\dim_{\R} \bigl( \Hom_{\scB}(\B_w, \D(\B_w)) \bigr)=1.
\]
On en d\'eduit que, sous cette hypoth\`ese, il existe 
%un unique morphisme $\B_w \to \D(\B_w)$ \`a scalaire pr\`es, et donc 
une \emph{unique} forme invariante non nulle sur $\B_w$ \`a scalaire pr\`es. D'apr\`es nos remarques pr\'ec\'edentes, elle est sym\'etrique et non d\'eg\'en\'er\'ee.

On peut v\'erifier par un calcul explicite (voir~\cite[Lemma~3.10]{ew1}) que si $w \in W$ et si $\uw$ est une expression r\'eduite pour $w$, pour tout $c \in \BS(\uw)^{-\ell(w)}$ non nul le scalaire $\langle c, \rho^{\ell(w)} \cdot c \rangle_{\BS(\uw)}$ est strictement positif. Comme, pour toute inclusion scind\'ee $\B_w \hookrightarrow \BS(\uw)$, la restriction $\B_w^{-\ell(w)} \to \BS(\uw)^{-\ell(w)}$ est un isomorphisme (voir la Remarque~\ref{rema:degre-Bw}\eqref{it:degre-Bw}), on obtient finalement, sous l'hypoth\`ese que la conjecture de Soergel est v\'erifi\'ee pour $w$, qu'il existe (\`a scalaire strictement positif pr\`es) une unique forme invariante $\langle -,-\rangle_{\B_w}$ sur $\B_w$ telle que $\langle c, \rho^{\ell(w)} \cdot c \rangle_{\BS(\uw)} >0$ pour $c \in \B_w^{-\ell(w)}$ non nul, et que cette forme est sym\'etrique et non d\'eg\'en\'er\'ee. On l'appellera \'egalement \emph{forme d'intersection}. (Les \'enonc\'es consid\'er\'es ci-dessous ne d\'ependent du choix de cette forme qu'\`a un scalaire positif pr\`es, de sorte qu'on peut la fixer de fa{\c c}on arbitraire.)

%Puisque $\dim(\B_w^{-\ell(w)})=1$, on peut donc fixer une telle forme $\langle -, - \rangle_{\B_w}$ \`a un scalaire strictement positif pr\`es en demandant que 

%--------------------------------------------------------------
\subsection{\'Enonc\'es}
\label{ss:enonces}
%--------------------------------------------------------------

Si $B$ est un bimodule de Soergel, on posera $\overline{B}:=B \otimes_R \R$, o\`u $\R$ est vu comme un $R$-module \`a gauche gradu\'e de la fa{\c c}on standard. Si $\langle -,- \rangle_B$ est une forme invariante sur $B$, elle induit une forme $\R$-bilin\'eaire gradu\'ee sur $\overline{B}$, qu'on notera de la m\^eme fa{\c c}on. Notons que, pour tout $x \in R^n$, le morphisme $x \cdot (-) : \overline{B} \to \overline{B}(n)$ est autoadjoint pour toute telle forme.

%Pour tout $x \in R$, la multiplication \`a gauche par $x$ sur $\overline{B}$ est auto-adjointe pour une telle forme.

Les r\'esultats principaux de~\cite{ew1} peuvent s'\'enoncer de la fa{\c c}on suivante.

\begin{theo}
\label{theo:ew1}
\begin{enumerate}
\item
\label{it:ew1-Soergel}
La Conjecture~{\rm \ref{conj:soergel}} est vraie pour la repr\'esentation $V$, c'est-\`a-dire $\mathrm{ch}_\Delta(\B_w)=\uH_w$ pour tout $w \in W$.
\item
\label{it:ew1-hL}
Pour tout $w \in W$, la donn\'ee de Lefschetz $(\overline{\B_w}, \rho \cdot (-))$ v\'erifie le th\'eor\`eme de Lefschetz difficile.
\item
\label{it:ew1-HR}
Pour tout $w \in W$, la donn\'ee de Hodge--Riemann $(\overline{\B_w}, \rho \cdot (-), \langle -,- \rangle_{\B_w})$ v\'erifie les relations de Hodge--Riemann.
\end{enumerate}
\end{theo}

Notons que, d'apr\`es le Lemme~\ref{lemm:HR-hL}, l'\'enonc\'e~\eqref{it:ew1-HR} du Th\'eor\`eme~\ref{theo:ew1} implique l'\'enonc\'e~\eqref{it:ew1-hL} ; cependant cette propri\'et\'e joue un r\^ole crucial dans la preuve de l'\'enonc\'e~\eqref{it:ew1-HR} et m\'erite donc d'\^etre \'enonc\'ee explicitement. D'autre part, l'\'enonc\'e~\eqref{it:ew1-Soergel} est n\'ecessaire pour fixer la forme $\langle -,- \rangle_{\B_w}$, et donc pour que l'\'enonc\'e~\eqref{it:ew1-HR} ait un sens.

Les trois propri\'et\'es \'enonc\'ees dans le Th\'eor\`eme~\ref{theo:ew1} sont d\'emontr\'ees simultan\'ement par Elias--Williamson, par r\'ecurrence sur la longueur de $w$. Pour expliquer cette preuve, on dira que \emph{la conjecture de Soergel}, resp.~\emph{le th\'eor\`eme de Lefschetz difficile}, resp.~\emph{les relations de Hodge--Riemann}, \emph{sont v\'erifi\'es en longueur $\leq n$} si l'\'enonc\'e~\eqref{it:ew1-Soergel}, resp.~\eqref{it:ew1-hL}, resp.~\eqref{it:ew1-HR}, du Th\'eor\`eme~\ref{theo:ew1} est v\'erifi\'e pour tout $w \in W$ tel que $\ell(w) \leq n$.

La premi\`ere \'etape de la preuve est le r\'esultat suivant.

\begin{prop}
\label{prop:ew1-etape1}
Soit $n \geq 0$, et supposons que la conjecture de Soergel, le th\'eor\`eme de Lefschetz difficile et les relations de Hodge--Riemann sont v\'erifi\'es en longueur $\leq n$. Supposons d'autre part que pour tout $w \in W$ tel que $\ell(w)=n$ et tout $s \in S$ tel que $ws>w$ la donn\'ee de Hodge--Riemann $(\overline{\B_w \B_s}, \rho \cdot (-), \langle -, - \rangle_{\B_w \B_s})$ v\'erifie les relations de Hodge--Riemann, o\`u $\langle -, - \rangle_{\B_w\B_s}$ est d\'efinie comme au Lemme~{\rm \ref{lemm:formes}} \`a partir des formes d'intersection sur $\B_w$ et $\B_s$. Alors la conjecture de Soergel, le th\'eor\`eme de Lefschetz difficile et les relations de Hodge--Riemann sont v\'erifi\'es en longueur $\leq n+1 $.
\end{prop}

\begin{proof}[Id\'ee de preuve]
Soit $x \in W$ avec $\ell(x)=n+1$, et \'ecrivons $x=ys$ avec $s \in S$ et $\ell(y)=n$.
D'apr\`es une id\'ee de Soergel (voir~\cite[Lemma~7.1(2)]{soergel-bim}), puisque la Conjecture~\ref{conj:soergel} est vraie en longueur $\leq \ell(y)$, pour d\'emontrer que $\mathrm{ch}_\Delta(\B_x)=\uH_x$ il suffit de montrer que pour tout $z < x$ la forme bilin\'eaire sur $\Hom_\scB(\B_z, \B_y \B_s)$ d\'efinie par la composition
\[
\Hom_\scB(\B_z, \B_y \B_s) \otimes_{\R} \Hom_\scB(\B_y \B_s, \B_z) \to \Hom_{\scB}(\B_z, \B_z) \overset{\eqref{eqn:Hom-BxBy}}{=} \R
\]
(o\`u $\Hom_\scB(\B_y \B_s, \B_z)$ est identifi\'e \`a $\Hom_\scB(\B_z, \B_y \B_s)$ via $f \mapsto f^*$, o\`u l'adjoint est pris par rapport aux formes d'intersection sur $\B_z$ et $\B_y \B_s$)
%, qui sont sym\'etriques et non d\'eg\'en\'er\'ees
est non d\'eg\'en\'er\'ee.

Elias--Williamson d\'emontrent alors (voir~\cite[Theorem~4.1]{ew1}) qu'il existe une injection
\[
\Hom_\scB(\B_z, \B_y \B_s) \hookrightarrow \Bigl(\overline{\B_y \B_s} \Bigr)^{-\ell(z)}
\]
dont l'image est contenue dans la partie primitive de $(\overline{\B_y \B_s})^{-\ell(z)}$ (c'est-\`a-dire le noyau de $\rho^{\ell(z)+1} \cdot (-)$) et qui est une isom\'etrie \`a un scalaire positif pr\`es, o\`u $\Hom_\scB(\B_z, \B_y \B_s)$ est muni de la forme bilin\'eaire consid\'er\'ee ci-dessus et $(\overline{\B_y \B_s})^{-\ell(z)}$ de la ``forme de Lefschetz'' $(x,y) \mapsto \langle \rho^{\ell(z)} x, y \rangle_{\B_y\B_s}$. Puisque la restriction d'une forme d\'efinie positive ou n\'egative \`a un sous-espace est non d\'eg\'en\'er\'ee, les relations de Hodge--Riemann sur $\overline{\B_y \B_s}$ impliquent alors que notre forme sur $\Hom_\scB(\B_z, \B_y \B_s)$ est non d\'eg\'en\'er\'ee, ce qui implique la conjecture de Soergel pour l'\'el\'ement $x$.

Puisque les relations de Hodge--Riemann sont v\'erifi\'ees sur $\overline{\B_y \B_s}$, le th\'eor\`eme de Lefschetz difficile est vrai pour cet espace par le Lemme~\ref{lemm:HR-hL}. Ceci implique le th\'eor\`eme de Lefschetz difficile pour son sous-espace $\overline{\B_x}$. Les relations de Hodge--Riemann se transmettent \'egalement clairement de $\overline{\B_y \B_s}$ \`a $\overline{\B_x}$, puisque la restriction de $\langle -,- \rangle_{\B_y\B_s}$ \`a $\B_x$ co\"incide avec $\langle -,- \rangle_{\B_x}$ \`a un scalaire strictement positif pr\`es (voir notamment le Lemme~\ref{lemm:restriction-forme}) et puisque le degr\'e minimal en lequel ces espaces sont non nuls est le m\^eme (\`a savoir $-n-1$), voir la Remarque~\ref{rema:degre-Bw}\eqref{it:degre-Bw}.
\end{proof}

La Proposition~\ref{prop:ew1-etape1} ram\`ene la preuve du Th\'eor\`eme~\ref{theo:ew1} \`a la preuve des relations de Hodge--Riemann pour $\overline{\B_w \B_s}$ quand $ws>w$, sous l'hypoth\`ese o\`u la conjecture de Soergel, le th\'eor\`eme de Lefschetz difficile et les relations de Hodge--Riemann sont connus en longueur $\leq \ell(w)$. 
%Les ingr\'edients principaux de cette preuve sont les complexes de Rouquier, qu'on introduit dans la partie suivante.
La preuve de cette propri\'et\'e utilise une factorisation de l'op\'erateur de Lefschetz (c'est-\`a-dire la multiplication par $\rho$) construite \`a partir des complexes de Rouquier, comme expliqu\'e dans le paragraphe suivant.

%--------------------------------------------------------------
\subsection{Complexes de Rouquier}
%--------------------------------------------------------------

Pour tout $s \in S$, on note $F_s$ le complexe
\[
\cdots \to 0 \to \B_s \xrightarrow{\sm_s} R(1) \to 0 \to \cdots \quad \in C^{\mathrm{b}}(\scB),
\]
o\`u $\B_s$ est en degr\'e cohomologique $0$ et le morphisme $\sm_s : \B_s \to R(1)$ est d\'efini par $r \otimes r' \mapsto rr'$. Pour tout mot $\uw = (s_1, \cdots, s_n)$ en $S$, on d\'efinit alors le complexe
\[
F_{\uw} := F_{s_1} \otimes_R \cdots \otimes_R F_{s_n}.
% \quad \in \Kb(\scB).
\]
Pour \'eviter les confusions avec la $\Z$-graduation des bimodules, on notera ${}^n \hspace{-1pt} F_{\uw}$ pour le terme en degr\'e cohomologique $n$ du complexe $F_{\uw}$.
La propri\'et\'e cruciale de ces complexes (qu'on n'utilisera pas dans la preuve ci-dessous) est due \`a Rouquier (voir~\cite[Proposition~9.4 et~\S 9.3]{rouquier}) et affirme que, si $\uw$ est une expression r\'eduite, alors l'image de $F_{\uw}$ dans la cat\'egorie homotopique $\Kb(\scB)$ ne d\'epend que de l'image de $\uw$ dans $W$, \`a isomorphisme canonique pr\`es. 

Une propri\'et\'e que nous utiliserons ci-dessous, et qui est beaucoup plus facile \`a v\'erifier (par exemple par r\'ecurrence sur la longueur), est que pour toute expression r\'eduite $\uw$ d'image $w$ dans $W$ on a
\begin{equation}
\label{eqn:coh-Rouquier}
\mathsf{H}^i(F_{\uw} \otimes_R \R) = \begin{cases}
\R(-\ell(w)) & \text{si $i=0$;} \\
0 & \text{sinon.}
\end{cases}
\end{equation}

L'\'enonc\'e suivant joue un r\^ole crucial dans la preuve pr\'esent\'ee au~\S\ref{ss:preuve-ew1}\footnote{Cet \'enonc\'e n'est pas utilis\'e sous cette forme dans la preuve originale de~\cite{ew1}. Il apparait explicitement dans~\cite{williamson,ew2} (o\`u les hypoth\`eses consid\'er\'ees ici sont connues d'apr\`es~\cite{ew1}), et permet de simplifier l\'eg\`erement la preuve de~\cite{ew1}, comme on le verra au~\S\ref{ss:preuve-ew1}.}. Dans cet \'enonc\'e, on notera $M$ pour le morphisme $\id_M$.

\begin{lemm}
\label{lemm:morphisme-Lefschetz}
Soit $n \geq 1$.
Supposons la conjecture de Soergel connue en longueur $\leq n$.
%, et les relations de Hodge--Riemann connues en longueur $\leq n-1$. 
Supposons de plus que pour tout $y \in W$ tel que $\ell(y) \leq n-2$ et tout $s \in S$ tel que $ys>y$, le triplet $(\overline{\B_y \B_s}, \rho \cdot (-), \langle -,- \rangle_{\B_y \B_s})$ v\'erifie les relations de Hodge--Riemann\footnote{Comme remarqu\'e au~\S\ref{ss:enonces}, cette hypoth\`ese entraine que les relations de Hodge--Riemann sont connues en longueur $\leq n-1$.\label{fn:HR}} (o\`u la forme invariante sur $\B_y\B_s$ est celle fournie par le Lemme~{\rm \ref{lemm:formes}} \`a partir des formes d'intersection sur $\B_y$ et $\B_s$). Soit $w \in W$ de longueur $n$, et soit $\uw$ une expression r\'eduite pour $w$.

Alors il existe un bimodule de Soergel $D$ muni d'une forme invariante sym\'etrique non d\'eg\'en\'er\'ee $\langle -,- \rangle_D$ et un morphisme $d : \B_w \to D(1)$ tels que :
\begin{enumerate}
\item
\label{it:morphisme-Lefschetz-1}
$D$ est une somme directe de bimodules de la forme $\B_z$ avec $z<w$ (sans d\'ecalage) ;
\item
\label{it:morphisme-Lefschetz-2}
le degr\'e minimal en lequel $D$ est non nul est $-n+1$, et $D$ est concentr\'e en degr\'es de la m\^eme parit\'e que $-n+1$ ;
\item
\label{it:Rouquier-HR}
le triplet $(\overline{D}, \rho \cdot (-), \langle -,-\rangle_D)$ v\'erifie les relations de Hodge--Riemann ;
\item
on a $d^* \circ d = \rho \B_w - \B_w w^{-1}(\rho)$, o\`u l'adjoint est pris par rapport \`a la forme d'intersection sur $\B_w$ et la forme $\langle -,- \rangle_D$ sur $D$ ;
\item
il existe des inclusions scind\'ees $\B_w \hookrightarrow {}^0 \hspace{-1pt} F_{\uw}$ et $D(1) \hookrightarrow {}^1 \hspace{-1pt} F_{\uw}$ telles que le diagramme suivant est commutatif:
\[
\xymatrix@C=2cm{
\B_w \ar[d] \ar[r]^-{d} & D(1) \ar[d] \\
{}^0 \hspace{-1pt} F_{\uw} \ar[r]^-{{}^0 \hspace{-1pt} d_{F_{\uw}}} & {}^1 \hspace{-1pt} F_{\uw}.
}
\]
\end{enumerate}
\end{lemm}

\begin{proof}
On proc\`ede par r\'ecurrence sur $n$. Si $n=1$ on a $w=s \in S$, et on peut choisir $D=R$ et $d=\sqrt{\langle \rho, \alpha_s^\vee \rangle}\sm_s$. Ici l'adjoint de $d$ est $\sqrt{\langle \rho, \alpha_s^\vee \rangle} \delta_s$, o\`u le morphisme $\delta_s : R \to \B_s(1)$ est d\'efini par
\[
\delta_s(r) = r \cdot \frac{1}{2}(\alpha_s \otimes 1 + 1 \otimes \alpha_s),
\]
et il est facile de voir que
\begin{equation}
\label{eqn:polynomial-sliding}
\rho \B_s - \B_s s(\rho) = \langle \rho, \alpha_s^\vee \rangle \delta_s \circ \sm_s.
\end{equation}

Supposons maintenant le r\'esultat v\'erifi\'e pour $n$, et v\'erifions-le pour $n+1$. Soient $w \in W$ de longueur $n$ et $\uw$ une expression r\'eduite pour $w$, et soit $s \in S$ tel que $ws>w$. On consid\`ere l'expression r\'eduite $\uw s$ de $ws$. Soient $D$, $\langle -, - \rangle_D$ et $d : \B_w \to D(1)$ comme dans l'\'enonc\'e (pour l'\'el\'ement $w$), et notons $a : \B_w \hookrightarrow {}^0 \hspace{-1pt} F_{\uw}$ et $b : D(1) \hookrightarrow {}^1 \hspace{-1pt} F_{\uw}$ des choix d'inclusions scind\'ees telles que $b \circ d = {}^0 \hspace{-1pt} d_{F_{\uw}} \circ a$. \'Ecrivons $D=D^\uparrow \oplus D^{\downarrow}$ o\`u $D^\uparrow$, resp.~$D^{\downarrow}$, est une somme d'objets de la forme $\B_z$ avec $zs>z$, resp.~avec $zs<z$. Puisque $\Hom_{\scB}(D^\uparrow, \D(D^{\downarrow})) = \Hom_{\scB}(D^\downarrow, \D(D^{\uparrow}))=0$ d'apr\`es~\eqref{eqn:Hom-BxBy} et nos hypoth\`eses, cette d\'ecomposition est automatiquement orthogonale pour $\langle -, - \rangle_D$.

D'apr\`es~\eqref{eqn:rho-positivite} on a $\langle w^{-1} \rho, \alpha_s^\vee \rangle > 0$, de sorte qu'on peut d\'efinir
%On consid\`ere alors 
le morphisme
\[
f := \begin{pmatrix}
\sqrt{\langle w^{-1} \rho, \alpha_s^\vee \rangle} \B_w \sm_s \\
d \B_s
\end{pmatrix} : \B_w \B_s \to \B_w(1) \oplus D \B_s (1).
\]
Il est facile de v\'erifier que si on munit $\B_w \B_s$ de la forme $\langle -,- \rangle_{\B_w\B_s}$ fournie par le Lemme~\ref{lemm:formes}, et $\B_w \oplus D \B_s$ de la somme de la forme d'intersection de $\B_w$ et de la forme $\langle -,- \rangle_{D \B_s}$ fournie par le Lemme~\ref{lemm:formes}, on a
\[
f^* = \begin{pmatrix} 
\sqrt{\langle w^{-1} \rho, \alpha_s^\vee \rangle} \B_w \delta_s & d^* \B_s
\end{pmatrix},
\]
de sorte qu'en utilisant une formule similaire \`a~\eqref{eqn:polynomial-sliding} on trouve que
\[
f^* \circ f = \rho \B_w \B_s - \B_w \B_s (ws)^{-1}(\rho).
\]
De plus, on a un diagramme commutatif
\begin{equation}
\label{eqn:diag-comm-Rouquier}
\vcenter{
\xymatrix@C=3cm{
\B_w \B_s \ar[d]_-{a\B_s} \ar[r]^-{f} & \B_w (1) \oplus D\B_s(1) \ar[d]^-{\left( \begin{smallmatrix} \frac{1}{\sqrt{\langle w^{-1} \rho, \alpha_s^\vee \rangle}} a(1) & 0 \\ 0 & b\B_s \end{smallmatrix} \right)} \\
{}^0 \hspace{-1pt} F_{\uw s} = {}^0 \hspace{-1pt} F_{\uw} \B_s \ar[r]^-{{}^0 \hspace{-1pt} d_{F_{\uw s}}=\left( \begin{smallmatrix} {}^0 \hspace{-1pt} F_{\uw s} \sm_s \\ {}^0 \hspace{-1pt} d_{F_{\uw}} \B_s \end{smallmatrix} \right)} & {}^1 \hspace{-1pt} F_{\uw s} = {}^0 \hspace{-1pt} F_{\uw}(1) \oplus {}^1 \hspace{-1pt} F_{\uw} \B_s.
}
}
\end{equation}

Le bimodule gradu\'e $\B_{ws}$ est un facteur direct de $\B_w\B_s$. La conjecture de Soergel en longueur $\leq n+1$ implique qu'on a $\B_w\B_s \cong \B_{ws} \oplus E$, o\`u $E$ est une somme de termes de la forme $\B_z$ (sans d\'ecalage), et que de plus une telle d\'ecomposition est automatiquement orthogonale pour la forme $\langle -,- \rangle_{\B_w\B_s}$ (voir notamment~\eqref{eqn:BwBs-perverse},~\eqref{eqn:Hom-BxBy} et~\eqref{eqn:D-Bw}). On peut supposer que la restriction de cette forme \`a $\B_{ws}$ est $\langle -,- \rangle_{\B_{ws}}$, et alors l'adjoint de l'inclusion $\B_{ws} \hookrightarrow \B_w\B_s$ est la projection $\B_w\B_s \twoheadrightarrow \B_{ws}$ orthogonale \`a $E$. Ainsi, si $g : \B_{ws} \to  \B_w (1) \oplus D\B_s(1)$ est la compos\'ee de $f$ et de cette inclusion, on a
\begin{equation}
\label{eqn:adjoint-f}
g^* \circ g = \rho \B_{ws} - \B_{ws} (ws)^{-1}(\rho),
\end{equation}
et un analogue du diagramme commutatif~\eqref{eqn:diag-comm-Rouquier}.
Pour conclure, il nous reste maintenant \`a ``simplifier'' le terme $\B_w (1) \oplus D\B_s(1)$ de sorte qu'il satisfasse les conditions de l'\'enonc\'e.

La commutativit\'e de~\eqref{eqn:diag-comm-Rouquier}, le fait que ${}^1 \hspace{-1pt} d_{F_{\uw s}} \circ {}^0 \hspace{-1pt} d_{F_{\uw s}}=0$ et l'injectivit\'e de $b$ impliquent que la compos\'ee de $g$ avec le morphisme
\begin{equation}
\label{eqn:morph-elimination-gaussienne}
\B_w (1) \oplus D\B_s(1) \xrightarrow{\left( \begin{smallmatrix} d(1) & D\sm_s(1) \\ 0 & {}^1 \hspace{-1pt} d_{F_{\uw}} \circ (b\B_s) \end{smallmatrix} \right)} D(2) \oplus {}^2 F_{\uw} \B_s
\end{equation}
est nulle.
Le bimodule $D \B_s$ se d\'ecompose en une somme orthogonale $D^\uparrow \B_s \oplus D^{\downarrow} \B_s$. Ici on a $D^{\downarrow} \B_s \cong D^\downarrow(-1) \oplus D^\downarrow (1)$ d'apr\`es le Lemme~\ref{lemm:BwBs}\eqref{it:ws-1}, et cette d\'ecomposition peut \^etre choisie de telle sorte que la compos\'ee $D^\downarrow (1) \hookrightarrow D^{\downarrow} \B_s \xrightarrow{D^\downarrow \sm_s} D^\downarrow(1)$ est l'identit\'e (voir la preuve de~\cite[Lemma~6.5]{ew1}). 
Ainsi il existe des facteurs directs de $\B_w (1) \oplus D\B_s(1)$ et $D(2) \oplus {}^2 F_{\uw} \B_s$ qui s'identifient \`a $D^{\downarrow} (2)$, et tels que le facteur de l'application consid\'er\'ee dans~\eqref{eqn:morph-elimination-gaussienne} sur ces termes s'identifie \`a $\id_{D^{\downarrow}(2)}$. 
Ceci implique que l'application $g$ s'\'ecrit comme la compos\'ee\footnote{Cet argument est un cas particulier du proc\'ed\'e d'\emph{\'elimination gaussienne} qui permet, dans un complexe, d'\'eliminer un facteur direct de la forme $M \xrightarrow{\id} M$ sans changer le complexe \`a homotopie pr\`es.} de son facteur $h : \B_{ws} \to \B_w(1) \oplus D^\uparrow \B_s(1) \oplus D^{\downarrow}$ par le morphisme 
$\B_w(1) \oplus D^\uparrow \B_s(1) \oplus D^{\downarrow} \to \bigl( \B_w(1) \oplus D^\uparrow \B_s(1) \oplus D^{\downarrow} \bigr) \oplus D^{\downarrow}(2)$ donn\'e par $\left( \begin{smallmatrix} \id \\ - \gamma \end{smallmatrix} \right)$, o\`u $\gamma$ est la compos\'ee de l'inclusion $\B_w(1) \oplus D^\uparrow \B_s(1) \oplus D^{\downarrow} \hookrightarrow \B_w (1) \oplus D\B_s(1)$, suivie de~\eqref{eqn:morph-elimination-gaussienne}, puis de la projection sur $D^{\downarrow}(2)$. On remplace alors $g$ par $h$, et on obtient un diagramme similaire \`a~\eqref{eqn:diag-comm-Rouquier} (o\`u le terme en haut \`a droite est maintenant $\B_w(1) \oplus D^\uparrow \B_s(1) \oplus D^{\downarrow}$, et o\`u la fl\`eche de droite a \'et\'e modifi\'ee de la fa{\c c}on appropri\'ee).

%Ainsi il existe des facteurs directs de ${}^1 \hspace{-1pt} F_{\uw s}$ et ${}^2 \hspace{-1pt} F_{\uw s}$ qui s'identifient \`a $D^{\downarrow} (2)$, et tel que le facteur de $d^1_{F_{\uw s}}$ sur ces termes s'identifie \`a $\id_{D^{\downarrow}(2)}$. On peut supprimer ces termes dans $F_{\uw s}$ par ``\'elimination Gaussienne''. Concr\`etement, dans notre cas, cela signifie que l'application $g$ s'\'ecrit comme la compos\'ee de son facteur $h : \B_{ws} \to \B_w(1) \oplus D^\uparrow \B_s(1) \oplus D^{\downarrow}$ par le morphisme 
%$\B_w(1) \oplus D^\uparrow \B_s(1) \oplus D^{\downarrow} \to \bigl( \B_w(1) \oplus D^\uparrow \B_s(1) \oplus D^{\downarrow} \bigr) \oplus D^{\downarrow}(2)$ donn\'e par $\left( \begin{smallmatrix} \id \\ - \gamma \end{smallmatrix} \right)$, o\`u $\gamma$ est la compos\'ee de l'inclusion de $\B_w(1) \oplus D^\uparrow \B_s(1) \oplus D^{\downarrow}$ dans un suppl\'ementaire fix\'e de $D^\downarrow (2)$ dans ${}^1 \hspace{-1pt} F_{\uw s}$, suivie de $d^1_{F_{\uw s}}$, puis de la projection de ${}^2 \hspace{-1pt} F_{\uw s}$ vers $D^{\downarrow}(2)$ (le long d'un suppl\'ementaire fix\'e). On remplace alors $g$ par $h$, et on obtient un diagramme similaire \`a~\eqref{eqn:diag-comm-Rouquier} (o\`u le terme en haut \`a gauche est maintenant $\B_w(1) \oplus D^\uparrow \B_s(1) \oplus D^{\downarrow}$, et o\`u la fl\`eche de droite a \'et\'e modifi\'ee de la fa{\c c}on appropri\'ee).

Maintenant, on remarque que $\Hom_{\scB}(\B_{ws}, D^\downarrow) = 0$ d'apr\`es~\eqref{eqn:Hom-BxBy}, de sorte qu'on doit avoir $h(\B_{ws}) \subset \B_w (1) \oplus D^\uparrow \B_s(1)$, et qu'on peut donc consid\'erer $h$ comme un morphisme $\B_{ws} \to \B_w (1) \oplus D^\uparrow \B_s(1)$. On a alors
\[
h^* \circ h = \rho \B_{ws} - \B_{ws} (ws)^{-1}(\rho) ;
\]
en effet, cela d\'ecoule de~\eqref{eqn:adjoint-f} et du fait que la composante de $g$ donn\'ee par un morphisme $\B_{ws} \to D^{\downarrow}$ est nulle et que la composante de $g^*$ donn\'ee par un morphisme $D^{\downarrow}(2) \to \B_{ws}(2)$ \'egalement (pour la m\^eme raison que ci-dessus).

% ; alors la commutativit\'e de~\eqref{eqn:diag-comm-Rouquier} et le fait que $d^1_{F_{\uw s}} \circ d^0_{F_{\uw s}}=0$ assurent que $f(\B_w \B_s) \subset \B_w (1) \oplus D^\uparrow \B_s(1) \oplus D^\downarrow$.

%Remarquons maintenant que la conjecture de Soergel en longueur $\leq \ell(w)$,~\eqref{eqn:D-Bw} et~\eqref{eqn:Hom-BxBy} assurent que $\B_w \B_s$ est une somme d'objets $\B_z$ (sans d\'ecalage). En utilisant encore une fois ces arguments, on en d\'eduit que toute d\'ecomposition $\B_w \B_s = \B_{ws} \oplus E$ est automatiquement orthogonale. De plus, on peut supposer que la restriction de $\langle -,- \rangle_{\B_w\B_s}$ \`a $\B_{ws}$ est la forme d'intersection de $ws$. Alors l'adjoint de l'inclusion $\B_{ws} \hookrightarrow \B_w \B_s$ est la projection $\B_w \B_s \twoheadrightarrow \B_{ws}$, de sorte que si $g$ est la restriction de $f$ \`a $\B_{ws}$ on a $g^* \circ g = \rho \B_{ws} - \B_{ws} (ws)^{-1}(\rho)$. De plus, l'image de $g$ est incluse dans $\B_w (1) \oplus D^\uparrow \B_s(1)$. En effet on a d\'ej\`a vu que l'image de $f$ est incluse dans $\B_w (1) \oplus D^\uparrow \B_s(1) \oplus D^\downarrow$. On a $\Hom_{\scB}(\B_{ws}, D^\downarrow) = 0$ d'apr\`es~\eqref{eqn:Hom-BxBy}, de sorte qu'on doit effectivement avoir $g(\B_{ws}) \subset \B_w (1) \oplus D^\uparrow \B_s(1)$. 

%Donc (quitte \`a changer des signes ??) on peut prendre le morphisme $\B_{ws} \to \B_w (1) \oplus D^\uparrow \B_s$.

Pour conclure, il reste \`a voir que notre module $\B_w \oplus D^\uparrow \B_s$ v\'erifie les conditions~\eqref{it:morphisme-Lefschetz-1}--\eqref{it:Rouquier-HR}.
Pour~\eqref{it:morphisme-Lefschetz-1}, cela d\'ecoule de la conjecture de Soergel en longueur $\leq n$ et~\eqref{eqn:BwBs-perverse}, qui montrent que
$D^\uparrow \B_s$ est une somme d'objets $\B_z$ avec $z < ws$. La condition~\eqref{it:morphisme-Lefschetz-2} est claire puisque le degr\'e minimal en lequel $\B_w$, resp.~$D^\uparrow$, est non nul est $-n$, resp.~$\geq -n+1$.
En ce qui concerne~\eqref{it:Rouquier-HR}, pour le facteur $\B_w$ cela d\'ecoule de nos hypoth\`eses (voir la Note~\eqref{fn:HR}). Pour $D^\uparrow \B_s$, on remarque qu'on a une d\'ecomposition orthogonale
\[
D^\uparrow = \bigoplus_{\substack{\ell(y) \leq n-1 \\ \ell(y) \equiv -n+1 \pmod 2 \\ ys>y}} V_y \otimes_\R \B_y
\]
o\`u chaque $\B_y$ est muni de sa forme d'intersection et $V_y := \Hom_{\scB}(\B_y, D^\uparrow)$ est muni d'une forme bilin\'eaire sym\'etrique $(-1)^{\frac{-n+1+\ell(y)}{2}}$-d\'efinie. Donc la propri\'et\'e voulue d\'ecoule de nos hypoth\`eses sur les bimodules $\B_y\B_s$.
\end{proof}

\begin{rema}
\label{rema:positivite-inverse}
Dans~\cite{ew1} les auteurs prouvent des r\'esultats plus forts sur les complexes $F_{\uw}$ (mais qui ne sont pas n\'ecessaires \`a la preuve du Th\'eor\`eme~\ref{theo:ew1}). Ils d\'emontrent notamment que pour tout $w \in W$ et toute expression r\'eduite $\uw$ pour $w$, si la conjecture de Soergel est connue pour tous les \'el\'ements $y \leq w$, il existe un facteur direct $F'_{\uw} \subset F_{\uw}$ tel que l'inclusion est un isomorphisme dans $\Kb(\scB)$, et tel que pour tout $i$, ${}^i \hspace{-1pt} F'_{\uw}$ est 
%un facteur direct de ${}^i \hspace{-1pt} F_{\uw}$et 
une somme d'objets de la forme $\B_z(i)$ avec $z \in W$ ; voir~\cite[Theorem~6.9]{ew1}. En utilisant cela et le Th\'eor\`eme~\ref{theo:ew1}, on peut d\'emontrer que si on \'ecrit $H_w = \sum_y g_{y,w} \cdot \uH_y$ avec $g_{y,w} \in \Z[\vv,\vv^{-1}]$, alors $(-1)^{\ell(w)-\ell(y)} g_{y,w} \in \Z_{\geq 0}[\vv]$ pour tous $y,w \in W$ (voir~\cite[Remark~6.10]{ew1}).

Cette propri\'et\'e a \'et\'e \'etendue \`a un cadre ``tordu'' par Gobet~\cite{gobet} pour d\'emontrer une conjecture combinatoire de Dyer.
\end{rema}

\subsection{Preuve du Th\'eor\`eme~{\rm \ref{theo:ew1}}}
\label{ss:preuve-ew1}
%--------------------------------------------------------------

La preuve du Th\'eor\`eme~\ref{theo:ew1} va se faire par r\'ecurrence sur $\ell(w)$. Elle n\'ecessite encore un ingr\'edient : la d\'eformation. En fait, pour $w \in W$, $s \in S$ et $\zeta \in \R_{\geq 0}$, on consid\`ere ``l'op\'erateur de Lefschetz d\'eform\'e''
\[
L^{w,s}_\zeta := (\rho \cdot (-)) \B_s + \B_w (\zeta \rho \cdot (-)) : \B_w \B_s \to \B_w \B_s(2).
\]
On notera de m\^eme le morphisme induit de $\overline{\B_w\B_s}$ vers $\overline{\B_w\B_s}(2)$.

Des calculs explicites et relativement \'el\'ementaires montrent les r\'esultats suivants (voir~\cite[Theorem~5.1]{ew1} pour~\eqref{it:Lzeta-1} et \cite[Theorem~6.19]{ew1} pour~\eqref{it:Lzeta-2}).

\begin{prop}
\label{prop:Lzeta}
\begin{enumerate}
\item
\label{it:Lzeta-1}
Soient $w \in W$ et $s \in S$.
Supposons la conjecture de Soergel et les relations de Hodge--Riemann v\'erifi\'ees pour $w$. Alors le triplet $(\overline{\B_w \B_s}, L_\zeta^{w,s}, \langle -,- \rangle_{\B_w \B_s})$ v\'erifie les relations de Hodge--Riemann si $\zeta \gg 0$.
\item
\label{it:Lzeta-2}
Supposons que $\zeta>0$, et soient $w \in W$ et $s \in S$. Si $ws<w$ et si le th\'eor\`eme de Leftschetz difficile est v\'erifi\'e pour $w$, alors $(\overline{\B_w \B_s}, L_\zeta^{w,s})$ v\'erifie le th\'eor\`eme de Lefschetz difficile.
\end{enumerate}
\end{prop}

On peut maintenant donner la preuve du Th\'eor\`eme~\ref{theo:ew1}. En fait on va d\'emontrer par r\'ecurrence sur $n$ que :
\begin{enumerate}
\item
\label{it:preuve-1}
la conjecture de Soergel, le th\'eor\`eme de Lefschetz difficile et les relations de Hodge--Riemann sont vrais en longueur $\leq n$;
\item
\label{it:preuve-2}
pour tout $\zeta \geq 0$, tout $w \in W$ et tout $s \in S$ tels que $ws>w$ et $\ell(ws) \leq n$, le triplet $(\overline{\B_w \B_s}, L_\zeta^{w,s}, \langle -,- \rangle_{\B_w \B_s})$ v\'erifie les relations de Hodge--Riemann (o\`u, comme d'habitude, la forme $\langle -,- \rangle_{\B_w\B_s}$ est celle fournie par le Lemme~\ref{lemm:formes} \`a partir des formes d'intersection sur $\B_w$ et $\B_s$).
%\item
%pour tout $\zeta > 0$, tout $w \in W$ et tout $s \in S$ tels que $ws<w$ et $\ell(w) \leq n$, le triplet $(\overline{\B_w \B_s}, L_\zeta^{w,s}, \langle -,- \rangle_{\B_w \B_s})$ v\'erifie les relations de Hodge--Riemann.
\end{enumerate}

Ces propri\'et\'es sont \'evidentes si $n \in \{0,1\}$.
On les suppose donc v\'erifi\'ees pour un $n \geq 1$.
%On fixe maintenant $n \geq 0$, et on suppose d\'emontr\'es la conjecture de Soergel, le th\'eor\`eme de Lefschetz difficile et les relations de Hodge--Riemann en longueur $\leq n$. 
%On fixe \'egalement $w \in W$ de longueur $n$, $s \in S$ tel que $ws>w$, et on cherche \`a d\'emontrer la conjecture de Soergel, le th\'eor\`eme de Lefschetz difficile et les relations de Hodge--Riemann pour $ws$.
D'apr\`es la Proposition~\ref{prop:ew1-etape1}, pour d\'emontrer~\eqref{it:preuve-1} en longueur $n+1$ il suffit de v\'erifier que pour tout $w \in W$ tel que $\ell(w)=n$ et tout $s \in S$ tel que $ws>w$ la donn\'ee de Hodge--Riemann $(\overline{\B_w \B_s}, \rho \cdot (-), \langle -, - \rangle_{\B_w \B_s})$ v\'erifie les relations de Hodge--Riemann. Puisque $\rho \B_w\B_s = L^{w,s}_0$, ceci est un cas particulier de la propri\'et\'e~\eqref{it:preuve-2} au rang $n+1$. Et pour d\'emontrer cette propri\'et\'e, d'apr\`es le Corollaire~\ref{coro:continuite} et la Proposition~\ref{prop:Lzeta}\eqref{it:Lzeta-1}, il suffit de v\'erifier que le couple $(\overline{\B_w \B_s}, L^{w,s}_\zeta)$ v\'erifie le th\'eor\`eme de Lefschetz difficile pour tout $\zeta \geq 0$.

Soient donc $w \in W$ de longueur $n$, $s \in S$ tel que $ws>w$, et $\zeta \geq 0$.
%La preuve se d\'ecompose en deux cas : $\zeta=0$ et $\zeta>0$. 
Consid\'erons tout d'abord le cas $\zeta=0$. Nos hypoth\`eses permettent d'appliquer le Lemme~\ref{lemm:morphisme-Lefschetz} \`a $w$ : soient $D$, $\langle -, - \rangle_D$ et $d : \B_w \to D(1)$ les donn\'ees fournies par ce lemme. D'apr\`es~\eqref{eqn:rho-positivite} on a $\langle w^{-1} \rho, \alpha_s^\vee \rangle > 0$, ce qui permet de consid\'erer le morphisme
%On consid\`ere le morphisme
\[
d' := \begin{pmatrix}
\sqrt{\langle w^{-1} \rho, \alpha_s^\vee \rangle} \B_w \sm_s \\
d \B_s
\end{pmatrix} : \B_w \B_s \to \B_w(1) \oplus D\B_s(1)
\]
(comme dans la preuve du Lemme~\ref{lemm:morphisme-Lefschetz}).
Ce morphisme est compatible avec la diff\'erentielle en degr\'e $0$ d'un complexe de Rouquier associ\'e \`a $ws$ (au m\^eme sens que dans le Lemme~\ref{lemm:morphisme-Lefschetz}), de sorte que~\eqref{eqn:coh-Rouquier} implique que $\overline{d'}:=d' \otimes_R \R$ est injectif en degr\'es $\leq n$. On a \'egalement, pour les formes bilin\'eaires appropri\'ees,
\begin{equation}
\label{eqn:adj-main-proof}
(d')^* \circ d' = \rho \B_w \B_s - \B_w \B_s (ws)^{-1}(\rho).
\end{equation}

Comme dans la preuve du Lemme~\ref{lemm:morphisme-Lefschetz}, on a une d\'ecomposition canonique et orthogonale $D = D^\uparrow \oplus D^\downarrow$ o\`u les facteurs directs de $D^\uparrow$, resp.~$D^\downarrow$, sont de la forme $\B_z$ avec $zs>z$, resp.~$zs<z$. On notera $\langle -,- \rangle_{D^\uparrow}$ et $\langle -,- \rangle_{D^\downarrow}$ les restrictions de $\langle -,- \rangle_D$ \`a ces facteurs. On a \'egalement une d\'ecomposition non canonique et non orthogonale $D^\downarrow \B_s \cong D^\downarrow (1) \oplus D^\downarrow (-1)$, qu'on peut choisir de telle sorte que la compos\'ee $D^\downarrow(1) \hookrightarrow D^\downarrow \B_s \xrightarrow{D^\downarrow \sm_s} D^\downarrow (1)$ est l'identit\'e. Dans la suite on fixe une telle d\'ecomposition. Notons que
\begin{equation}
\label{eqn:forme-Ddown}
\langle x,y \rangle_{D^\downarrow \B_s} = 0 \quad \text{pour $x,y \in D^\downarrow(1)$}
\end{equation}
(o\`u la forme $\langle -,- \rangle_{D^\downarrow \B_s}$ est celle fournie par le Lemme~\ref{lemm:formes}). En effet, la forme invariante $\langle -,- \rangle_{D^\downarrow \B_s}$ d\'efinit un morphisme
\[
D^\downarrow (1) \oplus D^\downarrow (-1) = D^\downarrow \B_s \to \D(D^\downarrow \B_s) = \D(D^\downarrow (1) \oplus D^\downarrow (-1)) \cong D^\downarrow (-1) \oplus D^\downarrow (1)
\]
(o\`u on a fix\'e un isomorphisme $\D(D^\downarrow) \cong D^\downarrow$). Le facteur de ce morphisme de $D^\downarrow (1)$ vers $\D(D^\downarrow(1)) \cong D^\downarrow (-1)$ doit \^etre nul par~\eqref{eqn:Hom-BxBy}, ce qui se traduit plus concr\`etement par~\eqref{eqn:forme-Ddown}.

Une fois ces consid\'erations \'etablies, on peut consid\'erer $d'$ comme un morphisme de $\B_w\B_s$ vers $\B_w(1) \oplus D^\uparrow \B_s(1) \oplus D^\downarrow(2) \oplus D^\downarrow$. On notera $d'=d'_1 + d'_2 + d'_3 + d'_4$ avec $d'_j$ un morphisme de $\B_w \B_s$ vers le $j$-i\`eme terme de cette d\'ecomposition, pour $j \in \{1,2,3,4\}$, et on pose $\overline{d'_j} := d'_j \otimes_R \R$.

%\bigskip
%
%De m\^eme, comme dans la preuve du Lemme~\ref{lemm:morphisme-Lefschetz}, ce morphisme prend ses valeurs\footnote{NON ! A REVOIR !} dans $\B_w(1) \oplus D^\uparrow \B_s(1) \oplus D^\downarrow$, donc peut \^etre consid\'er\'e comme un morphisme de $\B_w \B_s$ vers $\B_w(1) \oplus D^\uparrow \B_s(1) \oplus D^\downarrow$. \'Ecrivons $d'=d'_1 + d'_2 + d'_3$ avec $d'_k$ un morphisme de $\B_w \B_s$ vers le $k$-i\`eme terme de cette d\'ecomposition, pour $k \in \{1,2,3\}$. On notera $\overline{d'_k} := d'_k \otimes_R \R$.

Soit maintenant $i \geq 0$, et montrons que $(L_0^{w,s})^{\circ i} : (\overline{\B_w\B_s})^{-i} \to (\overline{\B_w \B_s})^i$ est un isomorphisme ou, de fa{\c c}on \'equivalente, est injective. Soit $x \in (\overline{\B_w \B_s})^{-i}$. Si $\overline{d'_4}(x) \neq 0$, alors du th\'eor\`eme de Lefschetz difficile pour $D^\downarrow$ (qui est connu par r\'ecurrence) on d\'eduit que $\rho^i \cdot \overline{d'_4}(x) \neq 0$, donc que $(L_0^{w,s})^{\circ i}(x) \neq 0$. 

On pose maintenant $V:=\overline{\B_w} \oplus \overline{D^\uparrow \B_s}$, $V':=\ker(\overline{d_4'})$, et on note $\varphi := \overline{d'_1} + \overline{d'_2} : V' \to V(1)$.
% la compos\'ee de la restriction de $\overline{d'}:=d' \otimes_R \R$ avec la projection sur $V(1)$ (de fa{\c c}on orthogonale \`a $\overline{D^\downarrow \B_s}(1)$. 
Nous allons v\'erifier que ces donn\'ees satisfont les conditions du Lemme~\ref{lemm:Lefschetz-recurrence}, ce qui ach\`evera la preuve du cas $\zeta=0$. En effet :
\begin{itemize}
\item
$V$ v\'erifie les relations de Hodge--Riemann gr\^ace \`a nos hypoth\`eses de r\'ecurrence, comme dans la preuve du Lemme~\ref{lemm:morphisme-Lefschetz} ;
\item
par le m\^eme argument ``d'\'elimination gaussienne'' que dans la preuve du Lemme~\ref{lemm:morphisme-Lefschetz}, on peut construire un diagramme commutatif similaire \`a~\eqref{eqn:diag-comm-Rouquier} dans lequel la premi\`ere ligne est $d'_1 \oplus d'_2 \oplus d'_4 : \B_w\B_s \to \B_w(1) \oplus D^\uparrow \B_s(1) \oplus D^\downarrow$ ; comme pour $\overline{d'}$ cela implique que $\overline{d'_1} \oplus \overline{d'_2} \oplus \overline{d'_4}$ est injectif en degr\'es $\leq n$ (donc a fortiori en degr\'es $\leq -1$), puis que $\varphi$ poss\`ede \'egalement cette propri\'et\'e ;
%par le m\^eme argument que pour $\overline{d'}$ (et en utilisant l'\'elimination gaussienne comme dans la preuve du Lemme~\ref{lemm:morphisme-Lefschetz}), $\varphi$ est injectif en degr\'es $\leq n$, donc a fortiori en degr\'es $\leq -1$;
\item
$\varphi$ commute aux op\'erateurs de Lefschetz car cette application est obtenue \`a partir d'un morphisme de bimodules ;
\item
pour $x,y \in V'$, on a
\begin{multline*}
\langle \varphi(x), \varphi(y) \rangle_V = \langle \overline{d'_1}(x),\overline{d'_1}(y) \rangle_{\B_w} + \langle\overline{d'_2}(x), \overline{d'_2}(y) \rangle_{D^\uparrow \B_s} \\
\overset{\eqref{eqn:forme-Ddown}}{=} \langle \overline{d'}(x),\overline{d'}(y) \rangle_{\B_w \oplus D\B_s} \overset{\eqref{eqn:adj-main-proof}}{=} \langle x,L_0^{w,s}(y) \rangle_{V'}.
\end{multline*}
\end{itemize}

%En appliquant le Lemme~\ref{lemm:Lefschetz-recurrence} \`a ces donn\'ees on obtient que la restriction de $(L_0^{w,s})^{\circ i}$ \`a $V'$ est injective, ce qui ach\`eve la preuve. 
%(Ici l'injectivit\'e de $\overline{d'_3}$ en degr\'es $\leq -1$ est assur\'ee par la compatibilit\'e avec un complexe de Rouquier et~\eqref{eqn:coh-Rouquier}.)

La preuve dans le cas $\zeta>0$ est bas\'ee sur les m\^emes id\'ees, mais dans ce cas on ne peut pas utiliser la d\'ecomposition $D^{\downarrow} \B_s \cong D^\downarrow (1) \oplus D^\downarrow(-1)$ car elle n'est pas compatible avec l'op\'erateur de Lefschetz d\'eform\'e. On consid\`ere encore le morphisme $d : \B_w \to D(1)$ fourni par le Lemme~\ref{lemm:morphisme-Lefschetz}, et on pose
\[
f := \begin{pmatrix}
\sqrt{\langle w^{-1} \rho, \alpha_s^\vee \rangle + \zeta \langle \rho, \alpha_s^\vee \rangle} \B_w \sm_s \\
d \B_s
\end{pmatrix} : \B_w \B_s \to \B_w(1) \oplus D\B_s(1).
\]
On peut alors v\'erifier que
\[
f^* \circ f = L_\zeta^{w,s} - \B_w\B_s (sw^{-1}(\rho) + \zeta s(\rho))
\]
(o\`u les formes invariantes sont choisies comme dans le Lemme~\ref{lemm:morphisme-Lefschetz})
et que, si $L$ est l'op\'erateur sur $\B_w(1) \oplus D\B_s(1)$ donn\'e par $\rho \B_w(1)$ sur $\B_w(1)$ et par $\rho D\B_s(1) + \zeta D\rho\B_s(1)$ sur $D\B_s(1)$, on a $\overline{f} \circ L^{w,s}_\zeta = L \circ \overline{f}$ o\`u $\overline{f}:=f \otimes_R \R$. On conclut encore une fois en utilisant le Lemme~\ref{lemm:Lefschetz-recurrence}. (Pour v\'erifier les relations de Hodge--Riemann sur $D\B_s$, on utilise les m\^emes arguments que dans la preuve du Lemme~\ref{lemm:morphisme-Lefschetz} pour se ramener au cas d'un bimodule $\B_y\B_s$. Dans ce cas, si $ys>y$ on utilise l'hypoth\`ese de r\'ecurrence. Et si $ys<y$ on utilise les deux \'enonc\'es de la Proposition~\ref{prop:Lzeta} et le Corollaire~\ref{coro:continuite}.)

%%%%%%%%%%%%%%%%%%%%%%%%%%%%%%%
\section{Applications en combinatoire et th\'eorie des repr\'esentations}
\label{sec:applications}
%%%%%%%%%%%%%%%%%%%%%%%%%%%%%%%

%------------------------------------------------------------------
\subsection{Positivit\'e des polyn\^omes de Kazhdan--Lusztig}
%------------------------------------------------------------------

Soit $(W,S)$ un syst\`eme de Coxeter.
La premi\`ere application des resultats de~\cite{ew1} est l'\'enonc\'e suivant, qui avait \'et\'e conjectur\'e par Kazhdan--Lusztig dans~\cite[Commentaire pr\'ec\'edant la D\'efinition~1.2]{kl}. (Le fait que la Conjecture~\ref{conj:soergel} implique ces \'enonc\'es avait d\'ej\`a \'et\'e remarqu\'e dans~\cite{soergel-bim}.)

\begin{theo}
\begin{enumerate}
\item
\label{it:positivite-1}
Pour tous $w,y \in W$, on a $h_{y,w} \in \Z_{\geq 0}[\vv]$.
\item
\label{it:positivite-2}
Pour tous $x,y,z \in W$, on a $\mu_{x,y}^z \in \Z_{\geq 0}[\vv,\vv^{-1}]$.
%\[
%\uH_w \cdot \uH_y \in \bigoplus_{z \in W} \Z_{\geq 0}[v,v^{-1}] \cdot \uH_z.
%\]
\end{enumerate}
\end{theo}

\begin{proof}
D'apr\`es la Proposition~\ref{prop:exemples-ref-fidele} il existe une repr\'esentation $V$ de $W$ qui v\'erifie les conditions de la Partie~\ref{sec:global}. Alors~\eqref{it:positivite-1} d\'ecoule de la Conjecture~\ref{conj:soergel} et de la d\'efinition de $\mathrm{ch}_\Delta$. Le point~\eqref{it:positivite-2} d\'ecoule \'egalement directement de la Conjecture~\ref{conj:soergel} et du fait que l'isomorphisme du Th\'eor\`eme~\ref{theo:categorification} respecte la structure d'alg\`ebre.
\end{proof}

%------------------------------------------------------------------
\subsection{Preuve alg\'ebrique de la conjecture de Kazhdan--Lusztig}
\label{ss:preuve-KL}
%------------------------------------------------------------------

Pour toutes les notions consid\'er\'ees dans cette sous-partie, on renvoie \`a~\cite{humphreys}.
Soit $\frg$ une alg\`ebre de Lie semisimple complexe, soit $\frb \subset \frg$ une sous-alg\`ebre de Borel, et soit $\frh \subset \frb$ une sous-alg\`ebre de Cartan. Soit $W$ le groupe de Weyl associ\'e, et $S \subset W$ le sous-ensemble des r\'eflexions simples, de sorte que $(W,S)$ est un syst\`eme de Coxeter (voir le~\S\ref{ss:def-Coxeter}). Si $\frh_\R \subset \frh$ est le sous-$\R$-espace vectoriel engendr\'e par les coracines, alors la conjecture de Soergel est connue pour l'action de $W$ sur $\frh_\R$ par le Th\'eor\`eme~\ref{theo:ew1}. On peut facilement en d\'eduire cette conjecture pour l'action de $W$ sur $\frh$, voir par exemple~\cite[Bemerkung~6.15]{soergel-bim}. On notera $w_0 \in W$ l'unique \'el\'ement de plus grande longueur.

Consid\'erons la cat\'egorie $\cO$ de Bernstein--Gelfand--Gelfand associ\'ee, et notons $\cO_0$ son bloc de $0$, c'est-\`a-dire la sous-cat\'egorie pleine constitu\'ee des modules sur lesquels l'annulateur du module trivial $\C$ dans le centre de $\mathcal{U}(\frg)$ agit de fa{\c c}on nilpotente.

Notons $\rho$ la demi-somme des racines positives de $\frg$ d\'etermin\'ees par $\frb$ et, pour tout $w \in W$, par $\mathsf{M}_w$ le module de Verma de plus haut poids $w^{-1}(\rho)-\rho$ et par $\mathsf{L}_w$ son unique quotient simple. Ces objets appartiennent \`a $\cO_0$, et $\{\mathsf{L}_w : w \in W\}$ est un syst\`eme de repr\'esentants des classes d'isomorphisme d'objets simples dans $\cO_0$. Il est bien connu que la cat\'egorie $\cO_0$ a assez d'objets projectifs, et on notera $P$ la couverture projective de $\mathsf{L}_{w_0}=\mathsf{M}_{w_0}$.

La deuxi\`eme application principale des r\'esultats de~\cite{ew1} est une preuve alg\'ebrique du r\'esultat suivant, qui avait \'et\'e conjectur\'e par Kazhdan--Lusztig dans~\cite[Conjecture~1.5]{kl}\footnote{La normalisation des modules $\mathsf{M}_w$ et $\mathsf{L}_w$ dans~\cite{kl} est diff\'erente de celle adopt\'ee ici. Pour comparer le Th\'eor\`eme~\ref{theo:KL} \`a~\cite[Conjecture~1.5]{kl} il faut remarquer que $h_{y,w} = h_{y^{-1}, w^{-1}} = h_{w_0 y w_0, w_0 w w_0}$ pour tous $y,w \in W$.} et d\'emontr\'e en utilisant des m\'ethodes g\'eom\'etriques par Brylinsky--Kashiwara~\cite{bk} et Be{\u\i}linson--Bernstein~\cite{bb} ind\'ependamment.

\begin{theo}
\label{theo:KL}
Pour tous $y,w \in W$, la multiplicit\'e de $\mathsf{L}_w$ comme facteur de composition de $\mathsf{M}_y$ est $h_{y,w}(1)$.
\end{theo}

Dans le reste de ce paragraphe, on donne un aper{\c c}u de cette preuve (inspir\'e de la partie ``alg\'ebrique'' de~\cite{soergel-kat}, mais l\'eg\`erement diff\'erente\footnote{Cette simplification de la preuve est rendue possible par les r\'esultats de~\cite{soergel-bim}, qui n'\'etaient pas disponibles au moment o\`u~\cite{soergel-kat} a \'et\'e publi\'e.}). Notons $C$ le quotient de $\mathrm{Sym}(\frh)$ par l'id\'eal engendr\'e par les \'el\'ements homog\`enes $W$-invariants de degr\'e strictement positif\footnote{Dans~\cite{soergel-kat}, Soergel appelle cette l'alg\`ebre \emph{l'alg\`ebre coinvariante} de $W$. Comme il l'a remarqu\'e plus tard, bien qu'il se soit impos\'e dans la litt\'erature, ce nom ne semble pas opportun dans la mesure o\`u il ne respecte pas l'usage habituel du mot ``coinvariant''.}. La premi\`ere \'etape consiste \`a construire un isomorphisme d'alg\`ebres canonique
\begin{equation}
\label{eqn:C-EndP}
C \xrightarrow{\sim} \Hom_{\cO_0}(P,P) \ ;
\end{equation}
voir~\cite[Endomorphismensatz~7]{soergel-kat} pour la preuve originale de cet \'enonc\'e (due \`a Soergel et bas\'ee sur des id\'ees de ``d\'eformation'' de la cat\'egorie $\cO$) et~\cite[\S 4]{bernstein} pour une preuve plus simple due \`a Bernstein. Une fois ce fait \'etabli, on peut consid\'erer le foncteur
\[
\mathbb{V} : \cO_0 \to C-\mathrm{Mod}, \quad M \mapsto \Hom_{\cO_0}(P,M).
\]

\begin{lemm}
Pour tous $M,N \in \cO_0$ projectifs, le foncteur $\mathbb{V}$ induit une injection
\[
\Hom_{\cO_0}(M,N) \hookrightarrow \Hom_C(\mathbb{V}(M), \mathbb{V}(N)).
\]
\end{lemm}

\begin{proof}
Puisque $N$ est projectif, 
%il admet une filtration dont les sous-quotients sont des modules de Verma (voir~\cite[\S 3.10]{humphreys}). Puisque le socle de $\Delta_w$ est $L_e$ pour tout $w \in W$ (voir~\cite[\S 4.2 et \S 4.8]{humphreys}), le socle de $N$ 
son socle
est une somme directe de copies de $\mathsf{L}_{w_0}$ (voir~\cite[Corollary~4.8]{humphreys}). En particulier, tout sous-module non nul de $N$ admet $\mathsf{L}_{w_0}$ comme facteur de composition, et n'est donc pas annul\'e par $\mathbb{V}$. Il s'ensuit que si $f : M \to N$ est non nul alors $\mathbb{V}(f)$ est non nul, ce qui prouve l'injectivit\'e voulue.
\end{proof}

On introduit ensuite, pour tout $s \in S$, le foncteur de ``croisement de mur''
\[
\vartheta_s : \cO_0 \to \cO_0
\]
donn\'e par une compos\'ee $T_\mu^0 T_0^\mu$ o\`u $\mu$ est un caract\`ere entier de $\frh$ qui appartient \`a l'hyperplan orthogonal \`a la racine associ\'ee \`a $s$ et passant par $-\rho$ (voir par exemple~\cite[\S 7.15]{humphreys}).
% la compos\'ee d'une translation vers un poids sur le mur associ\'e et d'une translation qui revient vers $0$. 
 On peut alors construire un isomorphisme de foncteurs
\[
\mathbb{V} \circ \vartheta_s \cong R \otimes_{R^s} \mathbb{V}
\]
en consid\'erant un analogue ``singulier'' de l'isomorphisme~\eqref{eqn:C-EndP} (voir par exemple~\cite[p.~422]{bernstein}) et sa compatibilit\'e avec~\eqref{eqn:C-EndP} via le foncteur de translation (voir par exemple~\cite[Theorem~8]{soergel}) puis des arguments g\'en\'eraux, voir~\cite[Theorem~10]{soergel}. Puisque les objets projectifs de $\cO_0$ se d\'eduisent de $\mathsf{M}_{e}$ (dont l'image par $\mathbb{V}$ est le module trivial $\C$) en appliquant les foncteurs $\vartheta_s$ et en prenant des facteurs directs, on en d\'eduit que la restriction de $\mathbb{V}$ aux objets projectifs prend ses valeurs dans les modules de Soergel (voir le~\S\ref{ss:modules-Soergel}), puis en comparant les dimensions d'espaces de morphismes (voir notamment~\eqref{eqn:Hom-fomula} et la Proposition~\ref{prop:modules-Soergel}) que cette restriction est pleinement fid\`ele et envoie la couverture projective $\mathsf{P}_w$ de $\mathsf{L}_w$ sur $\B_w \otimes_R \C$.

Si on note $\mathrm{Proj}(\cO_0)$ la cat\'egorie des objets projectifs dans $\cO_0$ et $\scB_0$ la cat\'egorie des modules de Soergel (non gradu\'es) on obtient alors un diagramme commutatif d'isomorphismes
\[
\xymatrix@C=2cm{
[\mathrm{Proj}(\cO_0)] \ar[rr]_-{\sim}^-{[\mathbb{V}]} \ar[rd]_{[M] \mapsto \sum_{w} (M : \mathsf{M}_w) \cdot w \hspace{0.5cm}}^-{\sim} && [\scB_0] \ar[ld]^{\sim} \\
& \Z[W] &
}
\]
o\`u la fl\`eche de droite est induite par l'isomorphisme du Th\'eor\`eme~\ref{theo:categorification}. (Dans la fl\`eche de gauche, on note $(M : \mathsf{M}_w)$ la multiplicit\'e de $\mathsf{M}_w$ dans une filtration standard de $M$ ; voir~\cite[Theorem~3.10]{humphreys}.) D'apr\`es la Conjecture~\ref{conj:soergel}, la fl\`eche de droite envoie la classe de $\B_w \otimes_R \C$ sur $\sum_y h_{y,w}(1) \cdot y$, ce qui montre que $(\mathsf{P}_w : \mathsf{M}_y) = h_{y,w}(1)$. On conclut en utilisant la formule de r\'eciprocit\'e, voir par exemple~\cite[\S 3.11]{humphreys}.

%%%%%%%%%%%%%%%%%%%%%%%%%%%%%%%
\section{Variantes : le cas local et le cas relatif}
\label{sec:variantes}
%%%%%%%%%%%%%%%%%%%%%%%%%%%%%%%

Les r\'esultats pr\'esent\'es en Partie~\ref{sec:global} forment la th\'eorie de Hodge ``globale'' des bimodules de Soergel. Dans~\cite{williamson} et~\cite{ew2}, les auteurs d\'eveloppent des versions ``locale'' et ``relative'' de cette th\'eorie, et en d\'eduisent des applications en combinatoire et en th\'eorie des repr\'esentations.

Dans cette partie on se place sous les m\^emes hypoth\`eses que dans la Partie~\ref{sec:global}.

%---------------------------------------------------------------------------------
\subsection{Cas local}
%---------------------------------------------------------------------------------

Dans cette sous-partie, on suppose de plus qu'il existe un \'el\'ement $\rho^\vee \in V$ tel que pour $w \in W$ et $s \in S$ on a
\[
\langle w(\rho^\vee), \alpha_s \rangle >0 \quad \Leftrightarrow \quad sw>w.
\]
(Cette hypoth\`ese est satisfaite dans les deux exemples de la Proposition~\ref{prop:exemples-ref-fidele}.)
Pour $w \in W$ et $B$ dans $\scB$, on note $\Gamma^w(B) := B \otimes_{R \otimes_\R R} \cO(\mathrm{Gr}(w))$. On note $\Gamma_w(-)$ pour $\Gamma_{\{w\}}(-)$, et on consid\`ere la compos\'ee
\[
i_w^B : \Gamma_w(B) \hookrightarrow B \twoheadrightarrow \Gamma^w(B),
\]
qu'on verra comme un morphisme de $R$-modules gradu\'es \`a gauche.
Il n'est pas tr\`es difficile de voir que si $Q$ est la localisation de $R$ en tous les \'el\'ements de la forme $y(\alpha_s)$ pour $s \in S$ et $y \in W$, alors $Q \otimes_R i_w^B$ est un isomorphisme. En particulier, si on fixe une ind\'etermin\'ee $z$ (de degr\'e $2$) et si on consid\`ere $\R[z]$ comme un $R$-module gradu\'e via $\xi \mapsto \langle \xi, \rho^\vee \rangle \cdot z$ pour $\xi \in V^*$, alors l'application
\[
\R[z] \otimes_R i_w^B : \R[z] \otimes_R \Gamma_w(B) \to \R[z] \otimes_R \Gamma^w(B)
\]
est injective. On note
\[
H_{x,y} := \mathrm{coker}(\R[z] \otimes_R i_x^{\B_y})(-1).
\]
Le premier r\'esultat important de~\cite{williamson} (conjectur\'e par Soergel, voir~\cite[Bemerkung~7.2]{soergel-bim}, et motiv\'e par ``l'exemple fondamental'' de Bernstein--Lunts~\cite[Chapter~14]{bl}) est que ces modules v\'erifient une version du th\'eor\`eme de Leftschetz difficile, comme suit.

\begin{theo}
\label{theo:williamson}
Pour tout $i \geq 0$, la multiplication par $z^i$ induit un isomorphisme $(H_{x,y})^{-i} \simto (H_{x,y})^i$.
\end{theo}

L'int\'er\^et principal de ce r\'esultat est qu'il permet de donner une preuve alg\'ebrique de la conjecture de Jantzen concernant la ``filtration de Jantzen'' des modules de Verma. (Voir~\cite[\S 8.12]{humphreys} pour une discussion de cette conjecture, et~\cite[\S\S1.4--1.5]{williamson} pour une discussion du lien entre le Th\'eor\`eme~\ref{theo:williamson} et cette conjecture, qui repose sur des travaux de Soergel~\cite{soergel-andersen} et K\"ubel~\cite{kubel1,kubel2}.)

La preuve du Th\'eor\`eme~\ref{theo:williamson} repose sur le m\^eme genre d'arguments que celle de~\cite{ew1}. Elle d\'emontre simultan\'ement une version des relations de Hodge--Riemann, qu'on explique maintenant. Les m\^emes consid\'erations que dans la Partie~\ref{sec:global} montrent qu'il existe une unique forme $\R$-bilin\'eaire gradu\'ee sym\'etrique non d\'eg\'en\'er\'ee $\langle -,- \rangle'_{\B_y} : \B_y \times \B_y \to R$ qui v\'erifie $\langle ra,b \rangle'_{\B_y} = \langle a,rb \rangle'_{\B_y} = r\langle a,b \rangle'_{\B_y}$ et $\langle ar,b \rangle'_{\B_y}=\langle a,br \rangle'_{\B_y}$ pour tous $a,b \in \B_y$ et $r \in R$. En restreignant cette forme \`a $\Gamma_x(\B_y)$, puis en tensorisant par $\R[z]$, et enfin en identifiant $\R[z,z^{-1}] \otimes_R \Gamma_x(\B_y)$ \`a $\R[z,z^{-1}] \otimes_R \Gamma^x(\B_y)$ via $\R[z,z^{-1}] \otimes_R i_x^{\B_y}$, on obtient une forme bilin\'eaire
\[
\langle -,- \rangle_y^x : (\R[z] \otimes_R \Gamma^x(\B_y)) \times (\R[z] \otimes_R \Gamma^x(\B_y)) \to \R[z,z^{-1}].
\]
Pour $i \geq 0$, on d\'efinit la \emph{partie primitive} $P^{-i} \subset (\R[z] \otimes_R \Gamma^x(\B_y))^{-i}$ comme l'orthogonal (pour $\langle -,- \rangle_y^x$) du sous-module de $\R[z] \otimes_R \Gamma^x(\B_y)$ engendr\'e par les \'el\'ements de degr\'e $<-i$. La version ``locale'' des relations de Hodge--Riemann d\'emontr\'ee dans~\cite{williamson} est la suivante.

\begin{theo}
Pour tout $i>0$, la forme bilin\'eaire sym\'etrique $(-,-)^i : P^{-i} \times P^{-i} \to \R$ d\'efinie par $(a,b)^i := z^i \langle a,b \rangle_y^x$ est $(-1)^{\ell(x) + \frac{-i-\min}{2}}$-d\'efinie, o\`u $\min$ est le degr\'e minimal en lequel $\R[z] \otimes_R \Gamma^x(\B_y)$ est non nul.
\end{theo}

%---------------------------------------------------------------------------------
\subsection{Cas relatif}
%---------------------------------------------------------------------------------

%Comme d\'ej\`a remarqu\'e dans la Partie~\ref{sec:global}, la conjecture de Soergel et la formule~\eqref{eqn:Hom-fomula} impliquent que
%\begin{equation}
%%\label{eqn:Hom-BxBy}
%\Hom_{\scB}(\B_x, \B_y) = \begin{cases}
%\R & \text{si $x=y$;} \\
%0 & \text{sinon}
%\end{cases}
%\quad \text{et que} \quad
%\Hom_{\scB}(\B_x, \B_y(n))=0 \text{ si $n<0$.}
%\end{equation}
Maintenant que la Conjecture~\ref{conj:soergel} est connue, les formules~\eqref{eqn:Hom-BxBy} s'appliquent pour tous $x,y \in W$.
On dira qu'un objet $B$ de $\scB$ est \emph{pervers}\footnote{Cette terminologie est bien s\^ur motiv\'ee par le cas des syst\`emes de Coxeter cristallographiques pr\'esent\'e au~\S\ref{ss:crystallog} : dans ce cas les bimodules pervers correspondent aux faisceaux pervers dans $\widetilde{\scB}$.} s'il est une somme directe d'objets $\B_x$ (sans d\'ecalage). Dans ce cas, comme d\'ej\`a utilis\'e dans la preuve du Lemme~\ref{lemm:morphisme-Lefschetz}, les propri\'et\'es~\eqref{eqn:Hom-BxBy} montrent que si on pose $V_x := \Hom_{\scB}(\B_x,B)$ il existe un isomorphisme \emph{canonique}
\[
\bigoplus_{x \in W} V_x \otimes_\R \B_x \simto B.
\]
De plus, tout objet $B$ de $\scB$ admet une filtration canonique (et fonctorielle)
\[
\cdots \subset \tau_{\leq i} B \subset \tau_{\leq i+1} B \subset \cdots
\]
par des inclusions scind\'ees telles que $\bigl( \tau_{\leq i} B / \tau_{\leq i-1} B \bigr) (i)$ est pervers pour tout $i \in \Z$ (et telle que $\tau_{\leq i} B = 0$ pour $i \ll 0$ et $\tau_{\leq i} B=B$ pour $i \gg 0$). On pose alors
\[
H^i(B) := \bigl( \tau_{\leq i} B/\tau_{\leq i-1} B \bigr)(i), \qquad H^i_z(B) := \Hom_{\scB}(\B_z, H^i(B)).
\]

Tout morphisme $f : B \to B'(j)$ induit des morphismes $H^i(f) : H^i(B) \to H^{i+j}(B')$, qui sont ``encod\'es'' par des morphismes d'espaces vectoriels $H^i_z(f) : H^i_z(B) \to H^{i+j}_z(B')$. D'autre part, si $B$ est dans $\scB$ et si $\langle -,- \rangle_B$ est une forme invariante sur $B$, les propri\'et\'es~\eqref{eqn:Hom-BxBy} montrent que cette forme induit un accouplement non d\'eg\'en\'er\'e $H^i(B) \times H^{-i}(B) \to R$. Une fois fix\'ees des formes d'intersection sur chaque $\B_x$, cet accouplement est ``encod\'e'' par des accouplements $H^i_z(B) \times H^{-i}_z(B) \to \R$, qu'on notera encore $\langle -,- \rangle_B$.

Le r\'esultat principal de~\cite{ew2} s'\'enonce de la fa{\c c}on suivante.

\begin{theo}
\label{theo:ew2}
Soient $x,y,z \in W$. L'op\'erateur $L : H_z^\bullet(\B_x\B_y) \to H_z^{\bullet+2}(\B_x\B_y)$ induit par le morphisme $\B_x \rho \B_y : \B_x \B_y \to \B_x \B_y(2)$ v\'erifie le th\'eor\`eme de Lefschetz difficile, c'est-\`a-dire que pour tout $i \geq 0$ il induit un isomorphisme
\[
L^{\circ i} : H_z^{-i}(\B_x\B_y) \simto H^i_z(\B_x\B_y).
\]
De plus, ces donn\'ees v\'erifient les relations de Hodge--Riemann au sens o\`u la restriction de la forme bilin\'eaire sym\'etrique sur $H_z^{-i}(\B_x\B_y)$ d\'efinie par $(x,y) \mapsto \langle x, L^i y \rangle_{\B_x\B_y}$ est $(-1)^{\frac{1}{2}(\ell(x)+\ell(y)-\ell(z)-i)}$-d\'efinie. (Ici, la forme $\langle -,- \rangle_{\B_x\B_y}$ est celle obtenue par le Lemme~{\rm \ref{lemm:formes}} \`a partir des formes d'intersection sur $\B_x$ et $\B_y$.)
\end{theo}

La preuve proc\`ede encore par r\'ecurrence (sur $\ell(x)+\ell(y)$, puis sur $\ell(y)$ \`a $\ell(x)+\ell(y)$ fix\'e) en utilisant les morphismes du Lemme~\ref{lemm:morphisme-Lefschetz}. Elle n\'ecessite de consid\'erer \'egalement des op\'erateurs de la forme
\[
a \cdot \B_x \rho \B_s \B_y + b \cdot \B_x \B_s \rho \B_y : \B_x\B_s\B_y \to \B_x\B_s\B_y(2)
\]
o\`u $a,b \in \R_{>0}$, $x,y \in W$ et $s \in S$, dont les auteurs montrent qu'ils v\'erifient encore des versions appropri\'ees du th\'eor\`eme de Leftschetz difficile et des relations de Hodge--Riemann.

La premi\`ere application du Th\'eor\`eme~\ref{theo:ew2} consid\'er\'ee dans~\cite{ew2} est de nature combinatoire.
Par d\'efinition, les coefficients $\mu_{x,y}^z$ consid\'er\'es au~\S\ref{ss:def-Coxeter} v\'erifient
\[
\mu_{x,y}^z = \sum_{i \in \Z} \dim \bigl( H^i_z(\B_x \B_y) \bigr) \cdot \vv^{-i}.
\]
On d\'eduit donc du Th\'eor\`eme~\ref{theo:ew2} le r\'esultat suivant, qui \'etait une conjecture ``de folklore'' dans la combinatoire de Kazhdan--Lusztig appel\'ee ``unimodalit\'e des constantes de structure'' ; voir par exemple~\cite{ducloux}.

\begin{coro}
Si pour tout $m \in \Z_{\geq 1}$ on pose
\[
[m] = \frac{\vv^m - \vv^{-m}}{\vv-\vv^{-1}} = \vv^{-m+1} + \vv^{-m+3} + \cdots + \vv^{m-3} + \vv^{m-1},
\]
alors pour tous $x,y,z \in W$ on a
\[
\mu_{x,y}^z \in \bigoplus_{m \geq 1} \Z_{\geq 0} \cdot [m].
\]
\end{coro}

Une autre application du Th\'eor\`eme~\ref{theo:ew2} est propos\'ee dans~\cite{ew2}. Elle concerne certaines cat\'egories mono\"idales associ\'ees aux cellules bilat\`eres dans $W$ d\'efinies par Lusztig. Le Th\'eor\`eme~\ref{theo:ew2} permet de montrer que ces cat\'egories sont rigides et pivotales ; voir~\cite[\S 5.2]{ew2} pour plus de d\'etails.

%%%%%%%%%%%%%%%%%%%%%%%%%%%%%%%%%%%%%%%
%%%%%%%%%%%%%%%%%%%%%%%%%%%%%%%%%%%%%%%

\end{document}